\documentclass[a4paper,12pt,reqno]{amsart}
\usepackage{amsmath,amssymb,amsthm, color}
\numberwithin{equation}{section}
\newtheorem{theorem}{Theorem}[section]
\newtheorem{proposition}[theorem]{Proposition}
\newtheorem{corollary}[theorem]{Corollary}
\newtheorem{lemma}[theorem]{Lemma}
\theoremstyle{definition}
\newtheorem*{remark}{Remark}
\newtheorem{example}{Example}[section]

\def\N{{\bf N}}
\def\R{{\bf R}}
\newcommand{\alp}{\alpha}

\newcommand{\dsp}{\displaystyle}
\newcommand{\lam}{\lambda} 

\begin{document}

\title[Classification of the structures of stable radial solutions]
{Classification of the structures of stable radial \\[1ex] 
solutions for semilinear  elliptic equations in $\R^N$ } 
\author{Yasuhito Miyamoto}
\address[Yasuhito Miyamoto]{Graduate School of Mathematical Sciences, 
The University of Tokyo, 
3-8-1 Komaba, Meguro-ku, Tokyo 153-8914, Japan}
\email{miyamoto@ms.u-tokyo.ac.jp}

\author{Y\={u}ki Naito}
\address[Y\={u}ki Naito]{
Department of Mathematics,  Hiroshima University \\
Higashi-Hiroshima, 739-8526, Japan}
\email{yunaito@hiroshima-u.ac.jp}
\date{\today}

\begin{abstract}
We study the stability of radial solutions of the semilinear elliptic equation
$\Delta u +f(u)=0$ in $\R^N$, where $N \geq 3$ and  $f$ is a general superciritical nonlinearity. 
We give a classification of the solution structures 
with respect to the stability of radial solutions, 
and establish criteria for the existence and nonexistence of stable radial solutions 
in terms of the limits of $f'(u)F(u)$ as $u \to 0$ or $\infty$, where $F(u) = \int^{\infty}_u 1/f(t)dt$. 
Furthermore, we show the relation between the existence of singular stable solutions 
and the solution structure. 
\end{abstract}



\maketitle

%%%%%%%%%%%%%%%%%%%%%%%%%%%%%%%%%%%%%%%%%%%%%%%%%5
\section{Introduction}

We consider the stability of radial solutions of the semilinear elliptic equation
\begin{equation}
	\Delta u +f(u)=0  \quad  \mbox{in} \ \R^N,
	\label{eq1.1}
\end{equation}
where $N \geq 3$ and  $f \in C^1(\R)$.
A solution $u \in C^2(\R^N)$ of (\ref{eq1.1}) is called stable if 
$$
	\int_{\R^N}(|\nabla \phi|^2 -f'(u)\phi^2)dx \geq 0
$$
for any $\phi \in C^1_0(\R^N)$. 
If $u$ is not stable, we say that $u$ is unstable. 
It is well known that stable solutions appear naturally in the variational viewpoint, and 
the study of stable solutions turns out to be an important and interesting problem. 
For a survey, we refer to \cite{CaCab, Dup, Tak} and the references there in. 
Though there is vast amount of literature about stable solutions in a bounded domain,
much less is known about stable solutions for general nonlinearities in $\R^N$.

A complete classification of stable solutions in $\R^N$ is provided for two important 
nonlinearities $f(u) = |u|^{p-1}u$, $p > 1$ and $f(u) = e^u$ 
by Farina \cite{Fara, Farb, Farc}.
In the case $f(u) = |u|^{p-1}u$ with $p > 1$, it was shown by \cite{Farb} that 
no nontrivial stable $C^2$ solution of (\ref{eq1.1}) exists if either  $N \leq 10$, $p < \infty$  
or $N \geq 11$, $p < p_{JL}$, where 
\begin{equation}
	p_{JL}= 1 + \frac{4}{N-4-2\sqrt{N-1}}.
	\label{eq1.2}
\end{equation}
On the other hand, if $N \geq 11$ and $p \geq p_{JL}$, 
Eq.~(\ref{eq1.1}) has smooth positive, bounded, stable radial solutions.
In the case $f(u) = e^u$, 
there is no stable $C^2$ solutions of (\ref{eq1.1}) if $2 \leq N \leq 9$ (see \cite{Farc}), and 
there exist stable radial $C^2$ solutions of (\ref{eq1.1}) if $N \geq 10$ (see \cite{DuFac, JoLu}).
For a general convex nondecreasing function $f > 0$, 
classification results for stable solutions of (\ref{eq1.1}) were provided by
Dupaigne-Farina \cite{DuFaa, DuFab, DuFac}.
In \cite {DuFaa, DuFab} they introduced a quantity 
\begin{equation}
	q(u) = \frac{f'(u)^2}{f(u)f''(u)},
	\label{eq1.3}
\end{equation}
and obtained several Liouville-type theorems 
for stable solutions of (\ref{eq1.1}) in terms of limit of $q(u)$ as $u \to 0$ and $\infty$.

In this paper we are interested in radial solutions of (\ref{eq1.1}). 
It was shown by Cabr\'e-Cappela \cite{CaCaa} and Villegas \cite{Vila} that 
every bounded radial stable solution of (\ref{eq1.1}) must be constant if $N \leq 10$. 
This result holds for any nonlinearity $f \in C^1(\R)$. 
Recently, this result is extended to non-radial case for 
locally Lipschitz $f \geq 0$ by Dupaigne-Farina \cite{DuFac}. 
The aim of this paper is to classify the solution structures 
with respect to the stability of radial solutions when $N \geq 11$, 
and establish criteria for the existence and nonexistence of stable radial solutions.  

Throughout this paper, we always assume that $f(u)$ in (\ref{eq1.1}) satisfies the following:
$$
	\left\{
	\begin{array}{c}
	\mbox{$f \in C^2(\R \setminus\{0\})\cap C^1(\R)$, $f$ is odd, i.e., $f(-u) = -f(u)$, 
	and hence $f(0) = 0$; }
	\\[2ex]
	\displaystyle
	f(u) > 0, \quad f'(u) > 0, \quad f''(u) > 0 
	\quad \mbox{and} \quad F(u) < \infty \quad \mbox{for} \ u > 0, \quad \mbox{where} 
	\\[2ex]
	\displaystyle
	 F(u) = \int^{\infty}_u \frac{ds}{f(s)}.
	\end{array}
	\right.
$$
Note that $F(u)$ is strictly decreasing for $u > 0$. 
Since $f \in C^1(\R)$ and $f(0) = 0$, $F(u)$ satisfies
$F(u) \to \infty$ as $u \to 0$.

For $\alpha > 0$, we denote by $u(r, \alpha)$ a unique solution of the problem 
\begin{equation}
	\left\{
	\begin{array}{c}
	\dsp
	u''+\frac{N-1}{r}u'+f(u)=0 \quad \mbox{for} \  r > 0,
	\\[2ex]
	u(0) = \alpha, \quad u'(0) = 0.
	\end{array}
	\right.
	\label{eq1.4}
\end{equation}
We denote by $r_0(\alpha)$ the first zero of $u(r, \alpha)$ for $r \geq 0$.
Define $r_0(\alpha) = \infty$ if $u(r, \alpha) > 0$ for all $r \geq 0$.

First we give a classification of the structures of stable solutions for (\ref{eq1.4}).

\begin{theorem}\label{thm1.1}
%{\bf Theorem 1.1.} \it 
If $u(r, \alpha)$ is stable, then $r_0(\alpha) = \infty$. 
For solutions to $(\ref{eq1.4})$, 
one of the following {\rm (I)-(III)} holds.
\begin{enumerate}
\item[{\rm (I)}] 
For any $\alpha > 0$, $u(r, \alpha)$ is unstable.

\item[{\rm (II)}] 
For any $\alpha > 0$, $u(r, \alpha)$ is stable. 
In this case, $u(r, \alpha)$ is strictly increasing in $\alpha > 0$ for each fixed $r > 0$.

\item[{\rm (III)}] 
There exists $\alpha^* \in (0, \infty)$ such that 
$u(r, \alpha)$ is stable if $\alpha \in (0, \alpha^*]$, 
and $u(r, \alpha)$ is unstable if $\alpha > \alpha^*$.
In this case, $u(r, \alpha)$ is strictly increasing in $\alpha \in (0, \alpha^*)$ for each fixed $r > 0$.
\end{enumerate}
\end{theorem}

\begin{remark}
In the case where $u(r, \alpha)$ is unstable, that is, 
in the cases (I) with $\alpha > 0$ or (III) with $\alpha > \alpha^*$, 
we do not know whether $r_0(\alpha) < \infty$ or $r_0(\alpha) = \infty$.
\end{remark}

As mentioned above, it was shown by \cite{CaCaa} and \cite{Vila} that (I) must hold if $N \leq 10$. 
To state our results on the existence and nonexistence of stable radial solutions when $N \geq 11$,  
we introduce the conditions on $f(u)$ as follows:
\begin{enumerate}
\item[{\rm (f1)}] 
There exists a finite limit 
\begin{equation}
	q_0 = \lim_{u \to 0}f'(u)F(u);
	\label{eq1.5}
\end{equation}

\item[{\rm (f2)}] 
There exists a finite limit 
\begin{equation}
	q_{\infty} = \lim_{u \to \infty}f'(u)F(u).
	\label{eq1.6}
\end{equation}
\end{enumerate}

One can check that $q_0 \geq 1$ if the limit in (f1) exists (see \cite[Lemma 2.1]{FuIoc}).
Assume that there exists a limit $\lim_{u \to 0+}q(u)$, where $q(u)$ is defined by (\ref{eq1.3}). 
Then (f1) holds and the limit of $q(u)$ as $u \to 0$ is also given by $q_0$ in (f1), since we obtain 
$$
	\lim_{u \to 0}f'(u)F(u) = \lim_{u \to 0}\frac{F(u)}{\frac{1}{f'(u)}} 
	= \lim_{u \to 0}\frac{f'(u)^2}{f(u)f''(u)}
$$ 
by L'Hospital's rule. 
As mentioned in Dupaine-Farina \cite{DuFab}, 
if there exists a limit $p_0 = \lim_{u \to 0}uf'(u)/f(u)$, 
then we have $1/p_0 + 1/q_0 = 1$. 
In fact, by L'Hospital's rule, we obtain 
$$
	\frac{1}{p_0} = \lim_{u \to 0}\frac{f(u)/f'(u)}{u} = 
	\lim_{u \to 0}\left(1-\frac{f(u)f''(u)}{f'(u)^2}\right) = 1 - \frac{1}{q_0}.
$$
Note that the constant $q_{\infty}$ in (f2) satisfies $q_{\infty} \geq 1$ (see  \cite[Remark 1.1]{FuIoa}). 
Similarly, if there exists a limit $\lim_{u \to \infty}q(u)$, then (f2) holds and 
the limit of $q(u)$ as $u \to \infty$ is also given by $q_{\infty}$ in (f2).  
Furthermore, we have $1/p_{\infty} + 1/q_{\infty} = 1$ if there exists a limit 
$p_{\infty} = \lim_{u \to \infty}uf'(u)/f(u)$.

For $N \geq 11$, define $q_{JL}$ by the H\"older conjugate of $p_{JL}$, 
defined by (\ref{eq1.2}), i.e., 
$$
	q_{JL} = \frac{N-2\sqrt{N-1}}{4}.
$$
For the existence of stable solutions, we obtain the following result.

\begin{theorem}\label{thm1.2}
%{\bf Theorem 1.2.} \it 
Let $N \geq 11$. 
Assume that there exist constants $q_1 \leq q_2$ and $0 < \ell \leq \infty$ satisfying 
$1 \leq q_1 \leq q_{JL}$,
\begin{equation}
	q_1 \leq f'(u)F(u) \leq q_2 \quad \mbox{for} \ 0 < u < \ell
	\label{eq1.7}
\end{equation}
and 
\begin{equation}
	q_2(2N-4q_1) \leq \frac{(N-2)^2}{4}.
	\label{eq1.8}
\end{equation}
Then $u(r, \alpha)$ is stable for $0 < \alpha < \ell$. 
Furthermore, if $(\ref{eq1.7})$ holds with $\ell = \infty$, 
then {\rm (II)} holds in Theorem $\ref{thm1.1}$. 
\end{theorem}

\begin{remark}
Let $N \geq 11$. Note that 
\begin{equation}
	q(2N-4q)< (N-2)^2/4 \quad \mbox{for} \ q < q_{JL}
	\label{eq1.9}
\end{equation}
and $q(2N-4q) = (N-2)^2/4$ for $q = q_{JL}$. 
In the case $f(u) = u^p$, we have $F(u)f'(u) \equiv q = p/(p-1)$. 
Then, letting $q_1 = q_2 = q$, we see that 
(\ref{eq1.7}) holds with $\ell = \infty$, and that (\ref{eq1.8}) holds if $N \geq 11$ and $q \leq q_{JL}$.
Note here that $1 < q = p/(p-1) \leq q_{JL}$ if and only if $p \geq p_{JL}$ when $N \geq 11$. 
By Theorem \ref{thm1.2}, we conclude that 
$u(r, \alpha)$ is stable for any $\alpha > 0$ if $N \geq 11$ and $p \geq p_{JL}$.
This result is same as that in \cite{Farb}, and hence the
conditions (\ref{eq1.7}) and (\ref{eq1.8}) are optimal in the case $f(u) = u^p$. 
\end{remark}

Assume that $N \geq 11$ and (f1) holds with $1 < q_0 < q_{JL}$. 
From (\ref{eq1.5}) and (\ref{eq1.9}) 
there exist constants $q_2 > q_1 \geq 1$ with $q_1 \leq q_0 \leq q_2$ and $\ell > 0$ such that 
(\ref{eq1.7}) and (\ref{eq1.8}) hold. 
By Theorem \ref{thm1.2}, we obtain the following corollary.

\begin{corollary}\label{cor1.3}
%{\bf Corollary 1.3.} \it 
Let $N \geq 11$. 
If {\rm (f1)} holds with $1 < q_0 < q_{JL}$, then $(\ref{eq1.1})$ has stable solutions, 
and hence {\rm (II)} or {\rm (III)} holds in Theorem $\ref{thm1.1}$.
\end{corollary}

We consider the condition that (II) holds in Theorem \ref{thm1.1}. 
Let $N \geq 11$. 
From (\ref{eq1.9}) we see that if constants $q_1$ and $q_2$ are sufficiently close to $q$ 
with $q < q_{JL}$, then $q_1$, $q_2$ satisfy (\ref{eq1.8}).

\begin{corollary}\label{cor1.4}
%{\bf Corollary 1.4.} \it 
Let $N \geq 11$. 
Assume that constants $q_1$ and $q_2$ satisfy 
 $1 \leq q_1 \leq q_2$ and $(\ref{eq1.8})$. 
If $f(u)$ satisfies
\begin{equation}
	q_1 \leq \frac{f'(u)^2}{f(u)f''(u)} \leq q_2 \quad 
	\mbox{for all} \ u > 0,
	\label{eq1.10}
\end{equation}
then $(\ref{eq1.7})$ holds with $\ell = \infty$, and hence 
{\rm (II)} holds in Theorem $1.1$. 
\end{corollary}

We will prove Corollary \ref{cor1.4} in Section 4 (see Lemma \ref{lem4.6}). 
As an example of $f$ satisfying (\ref{eq1.10}), we give 
$$
	f(u) = u^{p_1} + u^{p_2}, \quad \mbox{where} \ 
	p_1 = \frac{q_1}{q_1-1} \quad \mbox{and} \quad p_2 = \frac{q_2}{q_2-1}.
$$
For the details, see Example \ref{exm7.1}.

For the nonexistence of radial stable solutions, we obtain the following. 

\begin{theorem}\label{thm1.5}
%{\bf Theorem 1.5.} \it 
Let $N \geq 11$. 
If {\rm (f1)} holds with $q_0 > q_{JL}$, 
then $u(r, \alpha)$ is unstable for any $\alpha > 0$, that is, 
{\rm (I)} holds in Theorem $\ref{thm1.1}$.
\end{theorem} 

\begin{remark}
Assume that $N \geq 11$, and that there exists a limit 
\begin{equation}
	q_0 = \lim_{u \to 0}q(u),
	\label{eq1.11}
\end{equation}
where $q(u)$ is defined by (\ref{eq1.3}). 
It was shown by Dupaigne--Farina \cite{DuFab} that, 
if $q_0$ in (\ref{eq1.11}) satisfies  $q_0 > q_{JL}$, then 
any bounded nonnegative stable solution $u$ of $(\ref{eq1.1})$ must be $u \equiv 0$. 
As mentioned above, if the limit in (\ref{eq1.11}) exists, then 
(f1) holds with the same $q_0$ in (\ref{eq1.11}).
Thus the nonexistence of stable radial solutions is obtained 
under weaker assumption in Theorem \ref{thm1.5}. 
However, we do not know whether nonradial solutions of (\ref{eq1.1})  
are unstable under the condition of Theorem \ref{thm1.5}. 
\end{remark}

Next, we consider conditions for $u(r, \alpha)$ to be unstable 
for sufficiently large $\alpha > 0$. 

\begin{theorem}\label{thm1.6}
%{\bf Theorem 1.6.} \it 
Let $N \geq 11$. 
Assume that {\rm (f2)} holds with $q_{\infty} > q_{JL}$. 
Then there exists $\hat{\alpha} > 0$ such that 
$u(r, \alpha)$ is unstable for $\alpha> \hat{\alpha}$,
and hence {\rm (I)} or {\rm (III)} holds in Theorem $\ref{thm1.1}$.
\end{theorem}

Combining Corollary \ref{cor1.3} and Theorem \ref{thm1.6}, we obtain 
a sufficient condition for (III) to hold in Theorem \ref{thm1.1}.

\begin{corollary}\label{cor1.7}
%{\bf Corollary 1.7.} \it  
Let $N \geq 11$. 
Assume that {\rm (f1)} and {\rm (f2)} hold 
with $1 < q_0 < q_{JL} < q_{\infty}$. 
Then {\rm (III)} holds in Theorem $\ref{thm1.1}$.
\end{corollary}

To the best of our knowledge, Corollary \ref{cor1.7} is the first result 
which gives the conditions for the existence of both stable and unstable solutions.
For example, let
$$
	f(u) = \frac{u^{p_1}}{(1+u)^{p_2}} \quad \mbox{with} \ 
	p_1 > p_{JL} \quad \mbox{and} \quad 1 \leq p_1-p_2 < p_{JL}.
$$
Then {\rm (f1)} and {\rm (f2)} hold with $1 < q_0 < q_{JL} < q_{\infty}$, and 
hence (III) holds in Theorem \ref{thm1.1}. (See Example \ref{exm7.2} for the details.)  

Finally, we consider the relation between the existence of the singular stable solution 
and the solution structure for (\ref{eq1.1}).
By a singular radial solution $u^*$ of (\ref{eq1.1}) we mean that $u^* \in C^2(0, \infty)$ satisfies
\begin{equation}
	(u^*)'' + \frac{N-1}{r}(u^*)' + f(u^*) = 0 \quad \mbox{for} \ r > 0
	\label{eq1.12}
\end{equation}
and $u^*(x) \to \infty$ as $x \to 0$.

\begin{theorem}\label{thm1.8}
%{\bf Theorem 1.8.} \it 
Assume that {\rm (II)} holds in Theorem $1.1$, that is, $u(r, \alpha)$ is stable 
for all $\alpha > 0$. 
Then there exists a singular radial solution $u^*$ of $(\ref{eq1.1})$ satisfying 
\begin{equation}
	u^*(r) > 0 \quad \mbox{for all} \ r > 0, \quad f'(u^*(|\cdot|)) \in L^1_{\rm loc}(\R^N) \quad 
	\mbox{and}
	\label{eq1.13}
\end{equation}
\begin{equation}
	\int_{\R^N}(|\nabla \phi|^2 - f'(u^*(|x|))\phi^2) dx \geq 0 
	\quad \mbox{for any} \ \phi \in C^1_0(\R^N).
	\label{eq1.14}
\end{equation}
In particular, $u^*$ is obtained as the increasing limit of $u(\cdot, \alpha)$ as $\alpha \to \infty$.
\end{theorem}

Define $q_{S}$ by the H\"older conjugate of $p_{S} = (N+2)/(N-2)$, i.e., 
$q_{S} = (N+2)/4$. 
As the converse of Theorem \ref{thm1.8}, we obtain the following result.

\begin{theorem}\label{thm1.9}
%{\bf Theorem 1.9.} \it 
Assume that 
\begin{equation}
	f'(u)F(u) \leq q_S \quad \mbox{for sufficiently large} \ u.
	\label{eq1.15}
\end{equation}
If there exists a singular radial solution $u^*$ of $(\ref{eq1.1})$ satisfying 
$(\ref{eq1.13})$ and $(\ref{eq1.14})$, then {\rm (II)} holds in Theorem $1.1$.
\end{theorem}

\begin{remark}
(i) It was shown by \cite{MiNab} that the singular solution of (\ref{eq1.1}) is unique 
if the limit $q_{\infty} = \lim_{u \to \infty}q(u)$ 
satisfies $q_{\infty} < q_S$, where $q(u)$ is defined by (\ref{eq1.3}).

(ii) 
As a consequence of Theorems \ref{thm1.8} and \ref{thm1.9},  
(\ref{eq1.1}) has a singular solution $u^*$ satisfying (\ref{eq1.13}) and (\ref{eq1.14}) 
if and only if (II) holds, provided (\ref{eq1.15}) holds.
Note here that the type (II) holds 
if and only if the stable solution set $\{u(\cdot, \alpha)\}_{\alpha > 0}$ 
is unbounded in $L^{\infty}(\R^N)$. 
Analogue phenomena have been observed in the study of 
the nonlinear eigenvalue problem 
\begin{equation}
	\left\{
	\begin{array}{ll}
	\Delta u +\lambda g(u)=0 \quad & \mbox{in} \ \Omega,
	\\
	u>0 & \mbox{in} \ \Omega,
	\\
	u=0 & \mbox{on} \ \partial \Omega,
	\end{array}
	\right.
	\label{eq1.16}
\end{equation}
where $\Omega$ is a bounded and smooth domain in $\R^N$ with $N \geq 3$ and 
$\lam > 0$ is a parameter. 
In (\ref{eq1.16}), assume that $g$ is a $C^1$, nondecreasing, convex function defined 
on $[0, \infty)$ and satisfies $g(0) > 0$ and $\lim_{u \to \infty}\frac{g(u)}{u} = \infty$.
It is well known that there exists a finite extremal value $\overline{\lambda} > 0$ such that, 
if $0 < \lambda < \overline{\lambda}$, then (\ref{eq1.16}) admits a classical stable solution 
$u_{\lambda} \in C^2(\Omega)\cap C(\overline{\Omega})$, 
and if $\lam > \overline{\lambda}$, then (\ref{eq1.16}) has no classical solution. 
For $\lam = \overline{\lambda}$, a weak solution 
$\overline{u}$ of (\ref{eq1.16}) is obtained as the increasing limit of $u_{\lambda}$ 
as $\lambda \uparrow \overline{\lambda}$, and it is called the extremal solution. 
Thus, the extremal solution $(\overline{\lam}, \overline{u})$ is unbounded if and only if 
the stable solution $u_{\lam}$ satisfies 
$\|u_{\lam}\|_{L^{\infty}(\Omega)} \to \infty$ as 
$\lam \uparrow \overline{\lambda}$. 
It was shown by Brezis-V\'azques \cite[Theorem 3.1]{BrVa} that 
(\ref{eq1.16}) has a unbounded solution $(\lam^*, u^*)$ satisfying 
$u^* \in H^1_0(\Omega)$ and 
$$
	\int_{\Omega}\left(|\nabla\varphi|^2-\lambda^* f'(u^*)\varphi^2\right) dx \geq 0
	\quad \textrm{for all}\ \varphi \in C_0^1(\Omega),
$$
if and only if the extremal solution $(\overline{\lam}, \overline{u})$ is unbounded, that is, 
the stable solution set $\{u_{\lam}\}_{\lam \in (0, \overline{\lam})}$ is unbounded in $L^{\infty}(\R^N)$.
In the case of a ball, a sufficient condition of $f(u)$ for the unboundedness of 
stable solution set is given in \cite[Theorem 1.1]{MiNac}. 
\end{remark}

It should be mentioned that the stability of radial solutions 
is closely related to the separation phenomena of solutions. 
(See, e.g., \cite{Has} for H\'enon type equations.) 
Let $u(r, \alpha)$ and $u(r, \beta)$ be solutions of (\ref{eq1.4}) with 
$0 < \alpha < \beta$. 
We will see in Section 2 that $u(r, \alpha)$ is stable if $u(r, \alpha) < u(r, \beta)$ for all $r \geq 0$, 
and that $u(r, \beta)$ is unstable if the graphs of $u(r, \alpha)$ and $u(r, \beta)$ intersect 
each other in $(0, \infty)$. 
Using these properties, we will classify the solution structure 
into three types as in Theorem \ref{thm1.1}.   
An analogue classification can be observed in \cite{BaNaa} and \cite{BaNab} 
for radial solutions to the equations 
$\Delta u + K(|x|)u^p = 0$ and $\Delta u + K(|x|)e^u = 0$, respectively.

In order to show the existence and nonexistence of stable radial solutions 
in the proof of Theorems \ref{thm1.2}, \ref{thm1.5} and \ref{thm1.6}, 
we will utilize generalized scaling transformations introduced by \cite{Fuj, FuIoa, FuIob, Miy}. 
Let $g(u) = u^p$ with $p > 1$ or $g(u) = e^u$, and define 
$G(u) = \int^{\infty}_u 1/g(t)dt$.  
For functions $u$, $v$ and a constant $\lam > 0$, we consider a relation 
\begin{equation}
	F(u(x)) = \lam^2 G(v(y)) 
	\quad \mbox{with} \ x = \lam y \quad \mbox{for} \ x, y \in \R^N.
	\label{eq1.17}
\end{equation}
For a solution $v$ of the equation $\Delta v + g(v) = 0$, define $u(x)$ by (\ref{eq1.17}). 
Then $u(x)$ satisfies 
\begin{equation}
	\Delta_x u + f(u) + \frac{|\nabla_x u|^2}{f(u)F(u)}(q - f'(u)F(u)) = 0 
	\quad \mbox{for} \ x \in \R^N,
	\label{eq1.18}
\end{equation}
where $q = p/(p-1)$ if $g(u) = u^p$ and $q = 1$ if $g(u) = e^u$. 
From (\ref{eq1.18}) we will construct a subsolution of (\ref{eq1.1}) to obtain stable solutions
in the proof of Theorem \ref{thm1.2}.
On the other hand, for a solution $u(x)$ of (\ref{eq1.1}), define $v(y)$ by (\ref{eq1.17}).
Then $v(y)$ satisfies 
$$
	\Delta_y v +g(v) + 
	\frac{|\nabla_y v|^2}{g(v)G(v)}(f'(u(x))F(u(x))-q) = 0 \quad \mbox{for} \  y \in \R^N.
$$
We will show that, for a solution $u(\cdot, \alpha)$ of (\ref{eq1.4}), 
the corresponding solution $v$ converges 
to a solution of $\Delta w + g(w) = 0$ as $\alpha \to 0$ and $\alpha \to \infty$, respectively, 
in the proof of Theorems \ref{thm1.5} and \ref{thm1.6}. 
Using these properties, we obtain the nonexistence of stable radial solutions. 

The paper is organized as follows. 
In Section 2, we study the separation phenomena of radial solutions 
to obtain the structure of stable solutions, and give the proof of Theorem 1.1. 
In Section 3, we will introduce the generalized scale transformations, 
and in Sections 4 and 5, we show the existence and nonexistence of stable radial 
solutions, respectively. 
In Section 6, we investigate the relation between the existence of the singular stable solution
and the solution structure.
Finally, in Section 7, we give examples of nonlinearities 
satisfying the hypotheses in our results.

%%%%%%%%%%%%%%%%%%%%%%%%%%%%%%%%%%%%%%%
\section{Structure of radial stable solutions: Proof of Theorem \ref{thm1.1}}

First we will show some intersection properties of the solution $u(r, \alpha)$ of (\ref{eq1.4}). 

\begin{lemma}\label{lem2.1}
%{\bf Lemma 2.1.} \it 
Let $\alp, \beta > 0$, $\alpha \neq \beta$, and assume that 
$u(r, \alpha)$ has at least one zero in $(0, \infty)$, that is, $r_0(\alpha) < \infty$. 
Then $u(r,\beta)-u(r,\alp)$ has at least one zero in $(0, r_0(\alp))$. 
\end{lemma}

\begin{proof}
First we consider the case $\alpha < \beta$. 
Assume by contradiction that 
\begin{equation}
	u(r, \beta)-u(r, \alp) > 0 \quad \mbox{for} \ 0 \leq r < r_0(\alp).
	\label{eq2.1}
\end{equation}
Put $\phi_1(r) = u(r, \alp)$ and 
$\phi_2(r) = u(r, \beta)-u(r, \alp)$.  
Note here that the equation in (\ref{eq1.4}) can be written as 
$(r^{N-1}u')' + r^{N-1}f(u) = 0$ for $r > 0$. 
Then, for $i = 1, 2$, $\phi_i$ satisfies $\phi_i'(0) = 0$ and 
\begin{equation}
	(r^{N-1}\phi_i')' + r^{N-1}m_i \phi_i = 0 
	\quad \mbox{for} \ 0 < r < r_0(\alpha),
	\label{eq2.2}
\end{equation}
where $m_1(r) = f(u(r, \alpha))/u(r, \alpha)$ 
and $m_2$ satisfies  
$$
	f'(u(r, \alpha)) < m_2(r) < f'(u(r, \beta)) 
	\quad \mbox{for} \ 0 < r < r_0(\alpha),
$$
Since $f''(u) > 0$, we have $uf'(u) - f(u) > 0$ for $u > 0$, and hence 
$m_2(r) > m_1(r)$ for $0 < r < r_0(\alp)$ and  
\begin{equation}
	\phi_1(r_0(\alp)) = 0, \quad \phi_1'(r_0(\alp)) \leq 0 
	\quad \mbox{and} \quad \phi_2(r_0(\alp)) \geq 0.
	\label{eq2.3}
\end{equation}
Multiplying (\ref{eq2.2}) for $i = 1, 2$ by $\phi_2$ and $\phi_1$, 
respectively, 
and integrating them on $[0, r_0(\alp)]$, we obtain 
$$
	\left.\frac{}{}
	r^{N-1}(\phi_1'\phi_2 - \phi_1\phi_2')\right|_{r= 0}^{r_0(\alp)} = 
	-\int^{r_0(\alp)}_0 s^{N-1}(m_1(s)-m_2(s))\phi_1\phi_2ds > 0.
$$
On the other hand, from (\ref{eq2.3}), we have
$$
	\left.\frac{}{}
	r^{N-1}(\phi_1'\phi_2 - \phi_1\phi_2')\right|_{r= 0}^{r_0(\alp)} = 
	r_0(\alp)^{N-1}\phi_1'(r_0(\alp))\phi_2(r_0(\alp)) \leq 0.
$$
This is a contradiction.
Thus $u(r, \beta)-u(r, \alp)$ has at least one zero in $(0, r_0(\alp))$.

Next we consider the case $\alpha > \beta$. 
Assume by contradiction that $u(r, \alp)-u(r, \beta)$ has no zero in $(0, r_0(\alp))$. 
Then we have $r_0(\beta) \leq r_0(\alp)$ and 
\begin{equation}
	u(r,\alp) - u(r, \beta) > 0 \quad \mbox{for} \ 0 \leq r < r_0(\beta).
	\label{eq2.4}
\end{equation}
Exchanging the roles of $\alpha$ and $\beta$ in (\ref{eq2.4}), we obtain (\ref{eq2.1}), 
which leads to a contradiction by the argument above.
Thus $u(r, \alp)-u(r, \beta)$ has at least one zero in $(0, r_0(\alp))$. 
\end{proof}

Following the arguments in the proof of \cite[Theorem 1.1 and Lemma 4.1]{Has}, 
we obtain the next lemma. 

\begin{lemma}\label{lem2.2}
%{\bf Lemma 2.2. } \it 
Let $0 < \beta < \alpha$. 
\begin{enumerate}
\item[{\rm (i)}] 
If $0 < u(r, \beta) < u(r, \alpha)$ for all $r \geq 0$, then 
then $u(r, \beta)$ is stable.

\item[{\rm (ii)}] Assume that there exists $r_0 > 0$ such that
$0 < u(r, \beta) < u(r, \alpha)$ for $0 < r < r_0$ and 
$u(r_0, \beta) = u(r_0, \alpha)$. 
Then $u(r, \alpha)$ is unstable.
\end{enumerate}
\end{lemma}

\begin{proof}
(i) Set $w(r) = u(r, \alpha) -u(r, \beta)$. 
Then $w(r) > 0$ for $r \geq 0$. 
Since $f''(u) > 0$ for $u > 0$, 
by the mean value theorem, $w(|x|)$ with $x \in \R^N$ satisfies 
$$
	-\Delta w > f'(u(|x|, \beta))w \quad \mbox{in} \ \R^N.
$$
Take any $\phi \in C^1_0(\R^N)$. 
Multiplying the above by $\phi^2/w$ and integrating by parts, we obtain 
$$
	\int_{\R^N}f'(u(|x|, \beta))\phi^2 dx < -\int_{\R^N}\frac{\phi^2}{w}\Delta w dx 
	= \int_{\R^N}\left(\frac{2\phi}{w}\nabla \phi\cdot \nabla w 
	- \frac{\phi^2}{w^2}|\nabla w|^2\right)dx.
$$
Applying Young's inequality, we obtain 
$$
	\int_{\R^N}\frac{2\phi}{w}\nabla \phi\cdot \nabla wdx \leq 
	\int_{\R^N}\frac{\phi^2}{w^2}|\nabla w|^2 dx 
	+ \int_{\R^N}|\nabla \phi|^2dx.
$$
Thus we obtain  
$$
	\int_{\R^N}f'(u(|x|, \beta))\phi^2 dx < \int_{\R^N}|\nabla \phi|^2dx,
$$	
hence $u(r, \beta)$ is stable.

(ii) Set $w(r) = u(r, \alpha) -u(r, \beta)$. 
Then $w(r) > 0$ for $0 \leq r < r_0$ and $w(r_0) = 0$.
By the mean value theorem, $w(|x|)$ with $x \in \R^N$ satisfies 
$$
	-\Delta w < f'(u(|x|, \alpha))w \quad \mbox{for} \ |x| < r_0.
$$
Multiplying the above by $w$, and integrating by parts on $|x| < r_0$, we obtain 
$$
	\int_{B_{r_0}}|\nabla w|^2 dx < \int_{B_{r_0}} f'(u(|x|, \alpha)w^2 dx,
$$
where $B_{r_0} = \{x \in \R^N: |x| < r_0\}$. 
Extending $w(r) = 0$ for $r > r_0$, we have $w(|\cdot|) \in H^1(\R^N)$.
Since $C_0^1(\R^N)$ is dense in $H^1(\R^N)$, there exists $\phi \in C_0^1(\R^N)$ such that 
$$
	\int_{\R^N}|\nabla \phi|^2 dx < \int_{\R^N} f'(u(|x|, \alpha)\phi^2 dx.
$$
Then $u(r, \alpha)$ is unstable.
\end{proof}

\begin{lemma}\label{lem2.3}
%{\bf Lemma 2.3.} \it 
Assume that $u(r, \alpha)$ is stable. 
Then $r_0(\alpha) = \infty$ and, for any $\beta \in (0, \alpha)$, 
$u(r, \beta)$ is stable and $0 < u(r, \beta) < u(r, \alpha)$ for all $r \geq 0$.  
\end{lemma}

\begin{proof}
We first show $r_0(\alpha) = \infty$. 
Assume by contradiction that $r_0(\alpha) < \infty$. 
By Lemma~\ref{lem2.1}, for $\beta \in (0, \alpha)$, $u(r, \alpha)-u(r, \beta)$ has at 
least one zero in $(0, r_0(\alpha))$. 
Lemma \ref{lem2.2} (ii) implies that $u(r, \alpha)$ is unstable, 
which is a contradiction. Thus $r_0(\alpha) = \infty$. 

Let $\beta \in (0, \alpha)$. 
We show that $r_0(\beta) = \infty$. 
Assume by contradiction that $r_0(\beta) < \infty$. 
By Lemma \ref{lem2.1} (exchanging the notation of $\alpha$ and $\beta$), 
we see that $u(r, \alpha)-u(r, \beta)$ has at least one zero in $(0, r_0(\beta))$. 
Hence, by Lemma \ref{lem2.2} (ii), $u(r, \alpha)$ is unstable, 
which is a contradiction. Thus $r_0(\beta) = \infty$.  
Applying Lemma \ref{lem2.2} (ii) again, we obtain $u(r, \beta) < u(r, \alpha)$ for all $r \geq 0$. 
Lemma \ref{lem2.2} (i) implies that $u(r, \beta)$ is stable..
\end{proof}

\begin{lemma}\label{lem2.4}
%{\bf Lemma 2.4.} \it 
Assume that $u(r, \beta)$ is stable for all $\beta \in (0, \alpha)$ with some $\alpha > 0$. 
Then the following hold:
\begin{enumerate}
\item[{\rm (i)}] For any $0 < \beta_1 < \beta_2 < \alpha$, one has 
\begin{equation}
	0 < u(r, \beta_1) < u(r, \beta_2) \quad \mbox{for} \ r \geq 0;
	\label{eq2.5}
\end{equation}

\item[{\rm (ii)}] $u(r, \alpha)$ is stable. 
\end{enumerate}
\end{lemma}

\begin{proof}
Let $0 < \beta_1 < \beta_2 < \alpha$. 
By the assumption, $u(r, \beta_1)$ and $u(r, \beta_2)$ are stable. 
Applying Lemma \ref{lem2.3} with $\alpha = \beta_2$, we obtain (\ref{eq2.5}), 
and hence (i) holds.
Let $\beta \in (0, \alpha)$. 
Since $u(r, \beta)$ is stable, we have 
\begin{equation}
	\int_{\R^N}(|\nabla \phi|^2 - f'(u(|x|, \beta))\phi^2) dx \geq 0
	\label{eq2.6}
\end{equation}
for $\phi \in C^1_0(\R^N)$.
Assume that ${\rm supp}\ \phi \subset \{x \in \R^N: |x| \leq R\}$ with some $R > 0$.   
From (\ref{eq2.5}) it follows that $u(r, \beta)$ is increasing in $\beta \in (0, \alpha)$ 
for each fixed $r \geq 0$.
By the continuous dependence of the solution $u(r, \alpha)$ with respect to $\alpha > 0$, 
we see that $u(r, \beta) \to u(r, \alpha)$ as $\beta \to \alpha$ for each fixed $r \in [0, R]$. 
Let $\beta \to \alpha$ in (\ref{eq2.6}). 
Then, by the monotone convergence theorem, we obtain 
$$
	\int_{\R^N}(|\nabla \phi|^2 - f'(u(|x|, \alpha))\phi^2 dx \geq 0.
$$
Thus $u(r, \alpha)$ is stable, and hence (ii) holds. 
\end{proof}

\begin{proof}[Proof of Theorem \ref{thm1.1}] 
Lemma \ref{lem2.3} implies that $r_0(\alpha) = \infty$ if $u(r, \alpha)$ is stable.
Define a set $A \in (0, \infty)$ by 
$$
	A = \{\alpha > 0: \mbox{$u(r, \alpha)$ is stable}\}.
$$
Then (I) holds if $A = \emptyset$. Assume that $A \neq \emptyset$, and 
define $\alpha^* = \sup A$. Then $\alpha^* \in (0, \infty]$. 
Take any $\alpha \in (0, \alpha^*)$. 
By the definition of $\alpha^*$, there exists $\hat{\alpha} \in (\alpha, \alpha^*)$ 
such that $\hat{\alpha} \in A$, 
and hence $u(r, \alpha)$ is stable by Lemma \ref{lem2.3}.
Thus we obtain $(0, \alpha^*) \subset A$.
If $\alpha^* = \infty$, we have $A = (0, \infty)$ and (II) holds.
If $\alpha^* < \infty$, then $u(r, \alpha)$ is stable for all $\alpha \in (0, \alpha^*)$ 
and $u(r, \alpha^*)$ is also stable by Lemma \ref{lem2.4} (ii), and hence $A = (0, \alpha^*]$. 
Thus (III) holds.
By Lemma \ref{lem2.4} (i), $u(r, \alpha)$ is strictly increasing in $\alpha \in (0, \alpha^*)$ 
for each fixed $r > 0$.
\end{proof}

%%%%%%%%%%%%%%%%%%%%%%%%%%%%%%%%%%%%%%%%

\section{Generalized scaling transformations}

In this section, we will introduce generalized transformations of solutions to (\ref{eq1.1}).
Let $g(u) = u^p$ with $p > 1$ or $g(u) = e^u$. 
It is well known that the equation 
\begin{equation}
	\Delta u + g(u) = 0 \quad \mbox{in} \ \R^N
\label{eq3.1}
\end{equation}
is invariant under the transformation, with $\lam > 0$, 
\begin{equation}
	u_{\lam}(x) = \left\{
	\begin{array}{ll}
	\lam^{2/(p-1)}u(\lam x) \quad & \mbox{if} \ g(u) = u^p,
	\\[1ex]
	u(\lam x) + 2\log \lam \quad & \mbox{if} \ g(u) = e^u,
	\end{array}
	\right.
\label{eq3.2}
\end{equation}
that is, if $u$ solves (\ref{eq3.1}), then 
$u_{\lam}$ defined by (\ref{eq3.2}) also solves (\ref{eq3.1}). 
It was observed by Fujishima \cite{Fuj} that the both transformations in (\ref{eq3.2}) 
can be written in a unified way by 
\begin{equation}
	u_{\lam}(x) = G^{-1}(\lam^{-2}G(u(\lam x))),
\label{eq3.3}
\end{equation}
where 
\begin{equation}
	G(u) = \int^{\infty}_u \frac{dt}{g(t)} = 
	\left\{
	\begin{array}{ll}
	u^{1-p}/(p-1) \quad &\mbox{if} \ g(u) = u^p,
	\\[1ex]
	e^{-u} \quad &\mbox{if} \ g(u) = e^u.
	\end{array}
	\right.
\label{eq3.4}
\end{equation}

%The following generalized scale transformations of (\ref{eq3.3}) were 
%introduced in \cite{FuIoa, Miy}. 
%We will introduce generalized scale transformations of (\ref{eq3.3}).
We will generalize the transformation of (\ref{eq3.3}).
For $x \in \R^N$, we denote by 
$\nabla_x = (\partial/\partial x_1, \cdots, \partial/\partial x_N)$ and 
$\Delta_x = \partial^2/\partial x_1^2 + \cdots + \partial^2/\partial x_N^2$. 

\begin{lemma}\label{lem3.1}
%{\bf Lemma 3.1.} \it 
Assume that $g(u) = u^p$ with $p > 1$ or $g(u) = e^u$. 
Define $G(u)$ by $(\ref{eq3.4})$.
Let $q$ be a constant defined by  
$q = p/(p-1)$ if $g(u) = u^p$ and $q = 1$ if $g(u) = e^u$.
\begin{enumerate}
\item[{\rm (i)}] Let $u$ be a positive solution of $(\ref{eq1.1})$, and let $\lam > 0$. 
Define $v$ by 
$$
	v(y) = G^{-1}(\lam^{-2}F(u(x)) 
	\quad \mbox{with} \ 
	y = \frac{x}{\lam} \quad \mbox{for} \ x \in \R^N.
$$
Then $v(y)$ satisfies 
$$
	\Delta_y v +g(v) + 
	\frac{|\nabla_y v|^2}{g(v)G(v)}(f'(u(x))F(u(x))-q) = 0 \quad \mbox{for} \  y \in \R^N.
$$

\item[{\rm (ii)}] Let $v$ be a solution of $(\ref{eq3.1})$, and let $\lam > 0$. 
Assume that $v > 0$ in $\R^N$ if $g(u) = u^p$. 
Define $u$ by
$$
	u(x) = F^{-1}(\lam^2 G(v(y)) 
	\quad \mbox{with} \ x = \lam y \quad \mbox{for} \ y \in \R^N.
$$
Then $u(x)$ satisfies 
$$
	\Delta_x u + f(u) + \frac{|\nabla_x u|^2}{f(u)F(u)}(q - f'(u)F(u)) = 0 
	\quad \mbox{for} \ x \in \R^N.
$$
\end{enumerate}
\end{lemma}

These transformations were first introduced by Fujishima--Ioku \cite{FuIoa} 
and used in the study of the existence of a solution 
for semilinear parabolic equations. 
Later, in \cite{Miy} one of the author used in the study of bifurcation problem for 
semilinear elliptic equations. 
For the reader's convenience, we give the proof of Lemm \ref{lem3.1}.
First we show the following result for a general function $g$.

\begin{lemma}\label{lem3.2}
%{\bf Lemma 3.2.} \it 
Let $I = (0, \infty)$ or $\R$. 
Assume that $g \in C^1(I)$ satisfies 
$$
	g(u) > 0, \quad g'(u) > 0  
	\quad \mbox{and} \quad G(u) < \infty \quad \mbox{for} \ u \in I, 
	\quad \mbox{where} \quad
	G(u) = \int^{\infty}_u \frac{ds}{g(s)}.
$$
Let $u, v \in C^2(\R^N)$ satisfy $u(x) > 0$ and $v(x) \in I$ for $x \in \R^N$, 
and let $\lam > 0$. 
If $u$ and $v$ have a relation 
\begin{equation}
	F(u(x)) = \lam^{2}G(v(y)) \quad \mbox{with} \ x = \lam y 
	\quad \mbox{for} \ x, y \in \R^N,
\label{eq3.5}
\end{equation}
then $u(x)$ and $v(y)$ satisfy 
\begin{equation}
	\frac{1}{f(u)}\left(\Delta_x u + f(u) -\frac{f'(u)}{f(u)}|\nabla_x u|^2\right) 
	= 
	\frac{1}{g(v)}\left(\Delta_y v + g(v) -\frac{g'(v)}{g(v)}|\nabla_y v|^2\right) 
\label{eq3.6}
\end{equation}
and 
\begin{equation}
	\frac{|\nabla_x u|^2}{f(u(x))^2F(u(x))} = 
	\frac{|\nabla_y v|^2}{g(v(y))^2G(v(y))}. 
\label{eq3.7}
\end{equation}
\end{lemma}

\begin{proof}[Proof of Lemma $\ref{lem3.2}$] 
Differentiating the both sides of (\ref{eq3.5}) by $x$, we have 
\begin{equation}
	\frac{\nabla_x u}{f(u)} = \lam \frac{\nabla_y v}{g(v)}.
\label{eq3.8}
\end{equation}
Differentiating the above by $x$ again, we obtain 
$$
	\frac{f(u)\Delta_x u - f'(u)|\nabla_x u|^2}{f(u)^2} 
	= 
	\frac{g(v) \Delta_y v - g'(v)|\nabla_y v|^2}{g(v)^2}.
$$
Adding the both sides by $+1$, we obtain (\ref{eq3.6}).
From (\ref{eq3.8}) and (\ref{eq3.5}) we obtain (\ref{eq3.7}).
\end{proof}

\begin{proof}[Proof of Lemma $\ref{lem3.1}$] 
By Lemma \ref{lem3.2} we have 
\begin{equation}
	\begin{array}{rcl}
	\dsp
	\frac{1}{f(u)}\left(\Delta_x u + f(u)\right) 
	& = & \dsp
	\frac{1}{g(v)}\left(\Delta_y v + g(v) -\frac{g'(v)}{g(v)}|\nabla_y v|^2\right) 
	+ \frac{f'(u)F(u)}{g(v)^2G(v)}|\nabla_y v|^2
	\\[3ex]
	& = & \dsp
	\frac{1}{g(v)}\left(\Delta_y v + g(v) -\frac{g'(v)G(v)- f'(u)F(u)}{g(v)G(v)}
	|\nabla_y v|^2\right).
	\end{array}
\label{eq3.9}
\end{equation}
By the similar argument, we obtain 
\begin{equation}
	\frac{1}{g(v)}\left(\Delta_y v + g(v)\right) = 
	\frac{1}{f(u)}\left(\Delta_x u + f(u) -\frac{f'(u)F(u)- g'(v)G(v)}{f(u)F(u)}
	|\nabla_x u|^2\right).
\label{eq3.10}
\end{equation}
Note here that $g'(v)G(v) = p/(p-1) = q$ if $g(v) = v^p$ and 
$g'(v)G(v) = 1 = q$ if $g(v) = e^u$.
From (\ref{eq3.9}) and (\ref{eq3.10}) we obtain (i) and (ii), respectively.
\end{proof}

%%%%%%%%%%%%%%%%%%%%%%%%%%%%%%%%%%%%%%

\section{Existence of stable radial solutions}

In this section we show the existence of stable radial solutions of (\ref{eq1.1}), and 
give the proof of Theorem \ref{thm1.2}. 
First we show the following comparison result.

\begin{proposition}\label{prp4.1}
%{\bf Proposition 4.1. } \it
Let $\hat{v} \in C^2(0, \infty) \cap C^1[0, \infty)$ satisfy 
$\hat{v}(r) > 0$ for $r \geq 0$, $\hat{v}'(0) = 0$, and 
\begin{equation}
	\hat{v}'' + \frac{N-1}{r}\hat{v}' + f(\hat{v}) \geq 0 \quad \mbox{for} \ r > 0.
\label{eq4.1}
\end{equation}
Assume that  
\begin{equation}
	r^2 f'(\hat{v}(r)) \leq \frac{(N-2)^2}{4} 
	\quad \mbox{for} \ r > 0.
\label{eq4.2}
\end{equation}
Let $u(r, \alpha)$ be a solution of $(\ref{eq1.4})$. 
Then the following {\rm (i)} and {\rm (ii)} hold.

{\rm (i)} 
For $0 < \alpha <  \hat{v}(0)$, $r_0(\alpha) = \infty$ and 
$u(r, \alpha) < \hat{v}(r)$ for $r > 0$.

{\rm (ii)} 
If $0 < \alpha < \beta < \hat{v}(0)$, then 
$0 < u(r, \alpha) < u(r, \beta)$ for $r > 0$.
\end{proposition}

To prove Proposition \ref{prp4.1}, we need the following lemma.

\begin{lemma}\label{lem4.2}
%{\bf Lemma 4.2.} 
Let $w \in C^2(0, R]\cap C^1[0, R]$ satisfy 
\begin{equation}
	w'' + \frac{N-1}{r}w' + \frac{(N-2)^2}{4r^2}w > 0 
	\quad \mbox{for} \ 0 < r < R
\label{eq4.3}
\end{equation}
with some $R > 0$. 
Assume that $w'(0) = 0$ and $w(0) > 0$. 
Then $w(r) > 0$ for $0 \leq r \leq R$.
\end{lemma}

\begin{proof}
Assume by contradiction that there exists $R_1 \leq R$ such that   
\begin{equation}
	w(r) > 0 \quad \mbox{for} \ 0 \leq r < R_1 \quad 
	\mbox{and} \quad w(R_1) = 0.
\label{eq4.4}
\end{equation}
Define $z(r) = r^{-(N-2)/2}$ for $r > 0$. 
Then $z$ satisfies 
\begin{equation}
	z'' + \frac{N-1}{r}z' + \frac{(N-2)^2}{4r^2}z = 0 
	\quad \mbox{for} \ r > 0.
\label{eq4.5}
\end{equation}
Multiplying (\ref{eq4.3}) and (\ref{eq4.5}) by $z$ and $w$, respectively, we obtain 
\begin{equation}
	\bigl(r^{N-1}(w'(r)z(r) - z'(r)w(r))\bigr)' > 0 
	\quad \mbox{for} \  0 < r < R_1.
\label{eq4.6}
\end{equation}
Define $\Phi(r) = r^{N-1}(w'(r)z(r) - z'(r)w(r))$.
We have    
$$
	r^{N-1}w'(r)z(r) = o(r^{\frac{N}{2}}) \quad 
	\mbox{and} \quad 
	r^{N-1}z'(r)w(r) = O(r^{\frac{N}{2}-1}) \quad 
	\mbox{as} \ r \to 0.
$$
Then $\lim_{r \to 0}\Phi(r) = 0$.
Since $\Phi(r)$ is strictly increasing in $r \in (0, R_1)$ from (\ref{eq4.6}),  
we obtain $\Phi(R_1) = R_1^{N-1}w'(R_1)z(R_1) > 0$. 
This implies that $w'(R_1) > 0$.  
On the other hand, from (\ref{eq4.4}) we have $w'(R_1) \leq 0$, 
which is a contradiction. 
Thus we obtain $w(r) > 0$ for $0 \leq r \leq R$.
\end{proof}

\begin{proof}[Proof of Proposition $\ref{prp4.1}$] 
Let $0 < \alpha < \hat{v}(0)$. 
First we will show that 
\begin{equation}
	\hat{v}(r) > u(r, \alpha) \quad \mbox{for} \ 0 \leq r \leq r_0(\alpha).
\label{eq4.7}
\end{equation}
Put $w(r) = \hat{v}(r) - u(r, \alpha)$ for $0 \leq r \leq r_0(\alpha)$. 
Assume by contradiction that there exists $r_1 \in (0, r_0(\alpha)]$ such that 
\begin{equation}
	w(r) > 0 \quad \mbox{for} \ 0 < r < r_1 \quad 
	\mbox{and} \quad w(r_1) = 0.
\label{eq4.8}
\end{equation}
Since $\hat{v}$ satisfies (\ref{eq4.1}), we have 
$$
	w'' + \frac{N-1}{r}w' + V_1(r)w \geq 0
	\quad \mbox{for} \ 0 < r \leq r_1,
$$
where 
$$
	V_1(r) = \frac{f(\hat{v}(r))-f(u(r, \alpha))}
	{\hat{v}(r)-u(r, \alpha)} 
	\quad \mbox{for} \ 0 < r < r_1 
	\quad \mbox{and} \quad 
	V_1(r_1) = f'(\hat{v}(r_1)).
$$
Since $\hat{v}(r) > u(r, \alpha)  > 0$ for $0 < r < r_1$ and  
$f'(u)$ is strictly increasing for $u > 0$, we have 
$$
	V_1(r) < f'(\hat{v}(r)) \leq \frac{(N-2)^2}{4r^2} 
	\quad \mbox{for} \ 0 < r \leq r_1.
$$
Then $w$ satisfies 
$$
	w'' + \frac{N-1}{r}w' + \frac{(N-2)^2}{4r^2}w > 0 
	\quad \mbox{for} \ 0 < r < r_1.
$$
Then, by Lemma \ref{lem4.2} we obtain $w(r) > 0$ for $0 \leq r \leq r_1$. 
This contradicts (\ref{eq4.8}). Thus we obtain (\ref{eq4.7}).

Next we will show that $r_0(\alpha) = \infty$. 
Assume by contradiction that $r_0(\alpha) < \infty$. 
Let $\beta \in (\alpha, \hat{v}(0))$. 
Then, by the argument above, we obtain 
$u(r, \beta) < \hat{v}(r)$ for $0 \leq r \leq r_0(\beta)$. 
By Lemma \ref{lem2.1}, $u(r, \beta)- u(r, \alpha)$ has at least one zero in $(0, r_0(\alpha))$. 
Put $w(r) = u(r, \beta)- u(r, \alpha)$. 
Then there exists $r_1 \in (0, r_0(\alpha))$ such that 
\begin{equation}
	w(r) > 0 \quad \mbox{for} \ 0 \leq r < r_1 \quad 
	\mbox{and} \quad w(r_1) = 0.
\label{eq4.9}
\end{equation}
We see that $w$ satisfies 
\begin{equation}
	w'' + \frac{N-1}{r}w' + V_2(r) w = 0
	\quad \mbox{for} \ 0 < r \leq r_1,
\label{eq4.10}
\end{equation}
where 
\begin{equation}
	V_2(r) = \frac{f(u(r, \beta))-f(u(r, \alpha))}
	{u(r, \beta)-u(r, \alpha)} 
	\quad \mbox{for} \ 0 < r < r_1 
	\quad \mbox{and} \quad 
	V_2(r_1) = f'(u(r_1, \alpha)).
\label{eq4.11}
\end{equation}
We see that 
\begin{equation}
	0 < u(r, \alpha) \leq u(r, \beta) < \hat{v}(r) \quad \mbox{for} \ 0 < r \leq r_1.
\label{eq4.12}
\end{equation}
Since $f'(u)$ is strictly increasing for $u > 0$, we have 
\begin{equation}
	V_2(r) \leq f'(u(r, \beta)) < f'(\hat{v}(r))
	\leq \frac{(N-2)^2}{4r^2} 
	\quad \mbox{for} \ 0 < r \leq r_1.
\label{eq4.13}
\end{equation}
Then $w$ satisfies 
\begin{equation}
	w'' + \frac{N-1}{r}w' + \frac{(N-2)^2}{4r^2}w > 0 
	\quad \mbox{for} \ 0 < r < r_1.
\label{eq4.14}
\end{equation}
By Lemma \ref{lem4.2} we obtain $w(r_1) > 0$, which contradicts (\ref{eq4.9}). 
Thus we obtain $r_0(\alpha) = \infty$, and hence (i) holds.

To show (ii), put $w(r) = u(r, \beta) - u(r, \alpha)$ for $r \geq 0$.
We will show that $w(r) > 0$ for all $r \geq 0$. 
Assume by contradiction that, there exists $r_1 > 0$ such that (\ref{eq4.9}) holds.
Then $w$ satisfies (\ref{eq4.10}) with (\ref{eq4.11}). 
By (i), we have $u(r, \beta) < \hat{v}(r)$ for $r \geq 0$. 
By the similar argument to above, we obtain (\ref{eq4.12}), (\ref{eq4.13}) and (\ref{eq4.14}).
By Lemma 4.2 we obtain $w(r_1) > 0$, which contradicts (\ref{eq4.9}). 
Thus we obtain $w(r) > 0$ for all $r \geq 0$, hence (ii) holds.
\end{proof}

Let $q_1 \geq 1$ be the constant in Theorem \ref{thm1.2}.  
Define $g$ and $I \subset \R$ as 
$$
	\left\{
	\begin{array}{ll}
	g(u) = u^p \quad \mbox{with} \quad p = \frac{q_1}{q_1-1} 
	\quad \mbox{and} \quad I= (0, \infty) 
	& \mbox{if} \ q_1 > 1,
	\\[1ex]
	g(u) = e^u \quad \mbox{and} \quad I = \R 
	& \mbox{if} \ q_1 = 1.
	\end{array}
	\right.
$$ 
For $\alpha \in I$, denote by $w(r, \alpha)$ a solution of 
\begin{equation}
	\left\{
	\begin{array}{c}
	\dsp
	w'' + \frac{N-1}{r}w' + g(w) = 0, \quad r > 0,
	\\[2ex]
	w(0) = \alpha, \quad w'(0) = 0.
	\end{array}
	\right.
\label{eq4.15}
\end{equation}
We recall some properties of the solution $w(r, \alpha)$ from \cite{Wan, Tel}.

\begin{lemma}\label{lem4.3}
%{\bf Lemma 4.3.}\it 
(\cite[Proposition 3.7]{Wan}, \cite[Theorem 1.1]{Tel}) 
\begin{enumerate}
\item[{\rm (i)}] 
Assume that $q_1 > 1$. 
If $N \geq 11$ and $q_1 \leq q_{JL}$, i.e., $p \geq p_{JL}$, then 
$$
	w(r, \alpha) < W(r) :=  Lr^{-2/(p-1)} \quad \mbox{for} \ r > 0,
$$ 
where
$$
	L = \left\{\frac{2}{p-1}\left(
	N-2-\frac{2}{p-1}\right)\right\}^{1/(p-1)}.
$$
\item[{\rm (ii)}]  
Assume that $q_1 = 1$. If $N \geq 10$ then 
$$
	w(r, \alpha) < Z(r) := -2\log r + \log(2N-4) \quad \mbox{for} \ r > 0.
$$
\end{enumerate}
\end{lemma}

Define  $G$ by $(\ref{eq3.4})$. 
 From Lemma \ref{lem4.3}, we obtain the following.

\begin{lemma}\label{lem4.4}
%{\bf Lemma 4.4.} \it 
Let $N \geq 11$, and assume that $1 \leq q_1 \leq q_{JL}$.
Let $w(r, \alpha)$ be a solution of $(\ref{eq4.15})$ with $\alpha \in I$.  
Then $G(w(r, \alpha))$ satisfies
\begin{equation}
	G(w(r, \alpha)) > \frac{r^2}{2N-4q_1}
	\quad \mbox{for} \ r > 0.
\label{eq4.16}
\end{equation}
\end{lemma}

\begin{proof}
In the case $q_1 > 1$, by Lemma \ref{lem4.3} (i) we have 
$$
	G(w(r, \alpha)) > G(W(r)) = \frac{r^2}{2(N-2-\frac{2}{p-1})} = 
	\frac{r^2}{2N-4q_1}
	 \quad \mbox{for} \ r > 0.
$$
In the case $q_1 = 1$, by Lemma \ref{lem4.3} (ii) we have 
$$
	G(w(r, \alpha)) > G(Z(r)) = \frac{r^2}{2N-4q_1}
	\quad \mbox{for} \ r > 0.
$$
Thus (\ref{eq4.16}) holds in the both cases.
\end{proof}

Let $N \geq 11$, and let $1 \leq q_1 \leq q_{JL}$.
For a solution $w(r, \alpha)$ of (\ref{eq4.15}) with $\alpha \in I$, define 
\begin{equation}
	\hat{v}(r) = F^{-1}(G(w(r, \alpha))) \quad \mbox{for} \ r \geq 0.
\label{eq4.17}
\end{equation}
Then it is easy to see that 
$\hat{v} \in C^2(0, \infty) \cap C^1[0, \infty)$ and $\hat{v}$ satisfies 
$\hat{v}(r) > 0$ for $r \geq 0$, $\hat{v}'(0) = 0$. 
Furthermore, we have the following lemma.

\begin{lemma}\label{lem4.5}
%{\bf Lemma 4.5.} \it 
Assume that $\alpha \in I$ satisfies $F^{-1}(G(\alpha)) < \ell$, 
where $\ell \in (0, \infty]$ is the constant in Theorem $1.2$. 
Then $\hat{v}$, defined by $(\ref{eq4.17})$, satisfies $(\ref{eq4.1})$ and $(\ref{eq4.2})$ 
in Proposition $\ref{prp4.1}$.
\end{lemma}

\begin{proof}
By Lemma \ref{lem3.1} (ii) with $\lam = 1$, $\hat{v}(|x|)$ satisfies 
\begin{equation}
	\Delta \hat{v} + f(\hat{v}) 
	+ \frac{|\nabla \hat{v}|^2}{f(\hat{v})F(\hat{v})}(q_1 - f'(\hat{v})F(\hat{v})) = 0.
\label{eq4.18}
\end{equation}
Note that $F^{-1}(G(w))$ is increasing in $w > 0$. 
Since $w(r, \alpha) \leq \alpha$ for $r \geq 0$, we obtain 
$$
	0 < \hat{v}(r) \leq F^{-1}(G(w(0, \alpha))) = F^{-1}(G(\alpha)) < \ell
	\quad \mbox{for} \ r \geq 0.
$$
Then, from (\ref{eq1.7}), we have 
\begin{equation}
	q_1 \leq f'(\hat{v}(r))F(\hat{v}(r)) \leq q_2 \quad \mbox{for} \ r \geq 0.
\label{eq4.19}
\end{equation}
Since $\hat{v}$ satisfies (\ref{eq4.18}), we have (\ref{eq4.1}). 
By Lemma \ref{lem4.4} we have (\ref{eq4.16}). 
From (\ref{eq4.17}), (\ref{eq4.19}) and (\ref{eq1.8}) we obtain 
$$
	f'(\hat{v}(r)) \leq \frac{q_2}{F(\hat{v}(r))} = \frac{q_2}{G(w(r, \alpha))} 
	< \frac{q_2(2N-4q_1)}{r^2} 
	\leq \frac{(N-2)^2}{4r^2}
	\quad \mbox{for} \ r > 0.
$$
Thus $\hat{v}$ satisfies (\ref{eq4.2}).
\end{proof}

We are now in a position to prove Theorem \ref{thm1.2}.

\begin{proof}[Proof of Theorem $\ref{thm1.2}$] 
It suffices to show that $u(r, \beta)$ is stable for any $0 < \beta < \ell$. 
Take any $\beta \in (0, \ell)$. 
Note that $F^{-1}(G(\alpha)) \to \infty$ as $\alpha \to \infty$, and that 
$F^{-1}(G(\alpha)) \to 0$ as $\alpha \to 0$ if $I = (0, \infty)$ 
and as $\alpha \to -\infty$ if $I = \R$. 
Then we can take $\alpha_0 \in I$ such that $\beta < F^{-1}(G(\alpha_0)) < \ell$, 
and define $\hat{v}$ by (\ref{eq4.17}) with $\alpha = \alpha_0$. 
Then, by Lemma \ref{lem4.5}, $\hat{v}$ satisfies (\ref{eq4.1}) and (\ref{eq4.2}). 
Take $\beta_1$ such that $\beta < \beta_1 < \hat{v}(0) = F^{-1}(G(\alpha_0))$. 
Proposition \ref{prp4.1} implies that 
$0 < u(r, \beta) < u(r, \beta_1)$ for $r \geq 0$, 
which implies that $u(r, \beta)$ is stable by Lemma \ref{lem2.2}.
\end{proof}

Combining Theorem \ref{thm1.2} and the next lemma, 
we obtain Corollary \ref{cor1.4} immediately. 

\begin{lemma}\label{lem4.6}
%{\bf Lemma 4.6.} \it 
Assume that $f(u)$ satisfies
\begin{equation}
	q_1 \leq \frac{f'(u)^2}{f(u)f''(u)} \leq q_2 \quad 
	\mbox{for all} \ u > 0
\label{eq4.20}
\end{equation}
with some positive constants $q_1$ and $q_2$. 
Then $(\ref{eq1.7})$ holds with $\ell = \infty$.
\end{lemma}

\begin{proof}
From (\ref{eq4.20}) we have 
$$
	q_1\frac{f''(u)}{f'(u)^2} \leq \frac{1}{f(u)} \leq q_2\frac{f''(u)}{f'(u)^2}.
$$
Since $f''(u) > 0$ and $F(u) < \infty$, we have $f'(u) \to \infty$ as $u \to \infty$. 
Then, integrating the above on $[u, \infty)$ with $u > 0$, we obtain 
$$
	\frac{q_1}{f'(u)} \leq F(u) \leq \frac{q_2}{f'(u)} 
	\quad \mbox{for} \ u > 0.
$$
Thus $(\ref{eq1.7})$ holds with $\ell = \infty$.
\end{proof}

%%%%%%%%%%%%%%%%%%%%%%%%%%%

\section{Nonexistence of stable radial solutions}

\subsection{Proof of Theorem \ref{thm1.6}}
In this subsection, we assume that $N \geq 11$ and 
(f1) holds with $q_0 > q_{JL} > 1$.
Define $g(u) = u^p$ with $p = q_0/(q_0-1)$.
For $\sigma > 0$, let $z(s, \sigma)$ 
denote the solution of the problem
\begin{equation}
	\left\{
	\begin{array}{c}
	\dsp
	z''+\frac{N-1}{s}z'+ g(z) = 0 \quad \mbox{for} \  s > 0,
	\\[2ex]
	z(0) = \sigma, \quad z'(0) = 0.
	\end{array}
	\right.
\label{eq5.1}
\end{equation}
Define $G$ by (\ref{eq3.4}), and let $\sigma > 0$ be fixed. 
For $\alpha > 0$, define $\beta = \beta(\alpha) > 0$ as
\begin{equation}
	\frac{G(\sigma)}{G(1)} = \frac{F(\beta)}{F(\alpha)}.
\label{eq5.2}
\end{equation}
We note that $\beta(\alpha)$ is determined uniquely for each $\alpha > 0$, 
since $F(u)$ is strictly decreasing. 
We see that $\beta = \alpha$ if $\sigma = 1$.

For a solution $u(r, \alpha)$ of (\ref{eq1.4}), define $v(s, \alpha)$ as
\begin{equation}
	v(s, \alpha) = G^{-1}(\lam^{-2}F(u(r, \beta(\alpha)))) 
	\quad \mbox{with} \ 
	s = \frac{r}{\lam} \quad \mbox{and} \quad 
	\lam = \sqrt{\frac{F(\alpha)}{G(1)}},
\label{eq5.3}
\end{equation}
Then a relation between $u(r, \beta)$ and $v(s, \alpha)$ can be written as 
\begin{equation}
	\frac{G(v(s, \alpha))}{G(1)} = 
	\frac{F(u(r, \beta))}{F(\alpha)}.
\label{eq5.4}
\end{equation}
From (\ref{eq5.2}) and (\ref{eq5.4}) we have $v(0, \alpha) = \sigma$. 
Note that $p/(p-1) = q_0$. 
By Lemma \ref{lem3.1} (i), we  see that $v(s) = v(s, \alpha)$ satisfies 
\begin{equation}
	\left\{
	\begin{array}{c}
	\dsp
	v''+\frac{N-1}{s}v'+g(v) + 
	\frac{f'(u(r, \beta))F(u(r, \beta))-q_0}{g(v(s))G(v(s))}v'(s)^2 = 0 \quad \mbox{for} \  s > 0,
	\\[3ex]
	v(0) = \sigma \quad \mbox{and} \quad v'(0) = 0.
	\end{array}
	\right.
\label{eq5.5}
\end{equation}

Denote by $s_0(\sigma)$ the first zero of the solution $z(s, \sigma)$ of (\ref{eq5.1}). 
If $z(s, \sigma) > 0$ for all $s \geq 0$, define $s_0(\sigma) = \infty$. 
Since $F(u) \to \infty$ as $u \to 0$,
we have $\lam = \lam(\alpha) \to \infty$ as $\alpha \to 0$. 
We will show the following proposition.

\begin{proposition}\label{prp5.1}
%{\bf Proposition 5.1.} \it 
For any $S \in (0, s_0(\sigma))$, one has 
\begin{equation}
	v(s, \alp) \to z(s, \sigma) \quad \mbox{in} \ C[0, S] 
	\quad \mbox{as} \ \alpha \to 0.
\label{eq5.6}
\end{equation}
\end{proposition}

Take any $\tilde{S} \in (S, s_0(\sigma))$, and let $\tau = z(\tilde{S}, \sigma)$. 
Then $0 < \tau < z(S, \sigma) < \sigma$.
Define $\gamma = \gamma(\alpha) > 0$ by 
\begin{equation}
	\frac{G(\tau)}{G(1)} = 
	\frac{F(\gamma)}{F(\alpha)}.
\label{eq5.7}
\end{equation}
From (\ref{eq5.2}) and $\sigma > \tau$, we have $\gamma < \beta$. 
Since $f(u) > 0$ for $u > 0$, $u(r, \beta)$ is decreasing in $r > 0$, and hence 
either $u(r, \beta)$ has a zero for some $r_0 > 0$ or $u(r, \beta) > 0$ for all $r > 0$ and 
$u(r, \beta) \to 0$ as $r \to \infty$. 
Then there exists $r_{\gamma} > 0$ satisfying 
\begin{equation}
	\gamma = u(r_{\gamma}, \beta).
\label{eq5.8}
\end{equation}
Since $u(r, \beta)$ is decreasing in $r > 0$, we have 
\begin{equation}
	\gamma \leq u(r, \beta) \leq \beta \quad \mbox{for} \ 0 \leq r \leq r_{\gamma}.
\label{eq5.9}
\end{equation}
Define $s_{\gamma} = \lam^{-1}r_{\gamma}$. 
Then, from (\ref{eq5.4}), (\ref{eq5.7}) and (\ref{eq5.8}), 
we have $v(s_{\gamma}, \alpha) = \tau$.
Differentiating $G(v(s, \alpha)) = \lam^{-2}F(u(r, \beta))$ with respect to $s$, we have 
$$
	-\frac{v_s(s, \alpha)}{g(v(s, \alpha))} =
	-\frac{u_r(r, \beta)}{\lam f(u(r, \beta))}.
$$
Since $u_r(r, \beta) \leq 0$ for $0 \leq r \leq r_{\gamma}$, we have 
\begin{equation}
	v_s(s, \alpha) \leq 0 \quad \mbox{for} \ 0 \leq s \leq s_{\gamma}.
\label{eq5.10}
\end{equation}
Then we have 
\begin{equation}
	\tau \leq v(s, \alpha) \leq \sigma \quad \mbox{for} \ 0 \leq s \leq s_{\gamma}.
\label{eq5.11}
\end{equation}
Furthermore, we obtain the following lemma.

\begin{lemma}\label{lem5.2}
%{\bf Lemma 5.2.} \it
Define $v(s, \alpha)$ by $(\ref{eq5.3})$ for $0 \leq s \leq s_{\gamma}$. 
\begin{enumerate}
\item[{\rm (i)}] One has 
\begin{equation}
	\sup_{s \in [0, s_{\gamma}]}
	\frac{|f'(u(r, \beta))F(u(r, \beta))-q_0|}{g(v(s, \alpha))G(v(s, \alpha))} \to 0 
	\quad \mbox{as} \ \alpha \to 0.
\label{eq5.12}
\end{equation}
\item[{\rm (ii)}] 
For sufficiently small $\alpha > 0$,  
\begin{equation}
	0 \leq -v_s(s, \alpha) \leq C_1 s 
	\quad \mbox{for} \ 0 \leq s \leq s_{\gamma},
\label{eq5.13}
\end{equation}
where $C_1$ is a positive constant independent of $\alpha$ and $s$. 

\item[{\rm (iii)}] Denote $\gamma = \gamma(\alpha)$ by $(\ref{eq5.7})$. 
Then  $\liminf_{\alpha \to 0}s_{\gamma(\alpha)} > 0$.
\end{enumerate}
\end{lemma}

To prove Lemma \ref{lem5.2} (ii), we need the following lemma.

\begin{lemma}\label{lem5.3}
%{\bf Lemma 5.3.} \it 
Assume that $y \in C^1[0, T]$ satisfies 
$$
	y'(t) \leq A + By(t)^2 \quad \mbox{for} \ 0 \leq t \leq T
	\quad \mbox{and} \quad y(0) = 0,
$$
where $A, B$ and $T$ are positive constants satisfying 
$\sqrt{AB}T \leq \pi/4$.
Then 
\begin{equation}
	y(t) \leq 2At \quad \mbox{for} \ 0 \leq t \leq T.
\label{eq5.14}
\end{equation}
\end{lemma}

\begin{proof}[Proof of Lemma $\ref{lem5.3}$]
Define $Y(t) = \sqrt{B/A}y(t)$. Then $Y$ satisfies
$$
	Y'(t) \leq \sqrt{AB}(1+Y(t)^2) \quad  \mbox{for} \ 0 \leq t \leq T
	\quad \mbox{and} \quad Y(0) = 0.
$$
This implies that $Y(t) \leq \tan(\sqrt{AB}t)$ for $0 \leq t \leq T$. 
Since $\tan x \leq 2x$ for $0 \leq x \leq \pi/4$, we have 
$$
	Y(t) \leq 2\sqrt{AB}t \quad \mbox{for} \ 0 \leq t \leq T,
$$
which implies that (\ref{eq5.14}) holds.
\end{proof}

\begin{proof}[Proof of Lemma $\ref{lem5.2}$]
(i) From (\ref{eq5.2}) we see that $\beta(\alpha) \to 0$ as $\alpha \to 0$.
From (\ref{eq5.9}) we have, for any $r \in [0, r_{\gamma}]$,
$$
	0 < u(r, \beta) \leq \beta(\alpha) \to 0
	\quad \mbox{as} \ \alpha \to 0.
$$ 
The interval $0 \leq s \leq s_{\gamma}$ corresponds to $0 \leq r \leq r_{\gamma}$. 
Then, from (\ref{eq1.5}) we obtain 
$$
	\sup_{s \in [0, s_{\gamma}]}|f'(u(r, \beta))F(u(r, \beta))-q_0| \to 0 
	\quad \mbox{as} \ \alpha \to 0.
$$
Note that $g(u)G(u) = u/(p-1)$.
From (\ref{eq5.11}) we have 
$$
	g(v(s, \alpha))G(v(s, \alpha)) = \frac{v(s, \alpha)}{p-1} \geq \frac{\tau}{p-1} > 0 
	\quad \mbox{for} \ 0 \leq s \leq s_{\gamma}. 
$$
Thus we obtain (\ref{eq5.12}).
 
(ii) Define $y(s) = -v_s(s, \alpha)$. 
Then, from (\ref{eq5.10}) we have $y(0) = 0$ and $y(s) \geq 0$ for $0 \leq s \leq s_{\gamma}$.  
From (\ref{eq5.5}) and (\ref{eq5.11}), $y$ satisfies 
$$
	y'(s) \leq g(\sigma) + \frac{|f'(u(r, \beta))F(u(r, \beta))-q_0|}
	{g(v(s, \alpha))G(v(s, \alpha))}y(s)^2
	\quad \mbox{for} \ 0 < s \leq s_{\gamma}.
$$
Take $M > 0$ such that 
$$
	\sqrt{g(\sigma)M}s_{\gamma} \leq \frac{\pi}{4}. 
$$
From (\ref{eq5.12}) there exists $\alpha_0 > 0$ such that, if $0 < \alpha \leq \alpha_0$, 
$$
	\frac{|f'(u(r, \beta))F(u(r, \beta))-q_0|}{g(v(s, \alpha))G(v(s, \alpha))}
	\leq M
	\quad \mbox{for} \ 0 \leq s \leq s_{\gamma}.
$$
By Lemma \ref{lem5.3}, we obtain $y'(s) \leq C_1s$ for $0 \leq s \leq s_{\gamma}$ 
with $C_1 = 2g(\sigma)$. 
Thus we obtain (\ref{eq5.13}) if $0 < \alpha \leq \alpha_0$.

(iii) 
Integrating the both sides of (\ref{eq5.13}) on $[0, s_{\gamma(\alpha)}]$, we obtain 
$$
	0 \leq \sigma - \tau \leq \frac{C_1}{2}s_{\gamma(\alpha)}^2,
$$
which implies that 
$\liminf_{\alpha \to 0}s_{\gamma(\alpha)} > 0$.
\end{proof}

\begin{lemma}\label{lem5.4}
%{\bf Lemma 5.4.} \it 
Let $\{\alp_n\}$ be a sequence satisfying $0 < \alp_n \to 0$ as $n \to \infty$. 
Denote by $\gamma_n = \gamma(\alpha_n)$ and $s_{\gamma_n} = s_{\gamma(\alpha_n)}$, 
where $\gamma(\alpha)$ is defined by $(\ref{eq5.7})$.
Assume that 
\begin{equation}
	\liminf_{n \to \infty}s_{\gamma_n} > \hat{S}
\label{eq5.15}
\end{equation}
with some $\hat{S} > 0$. 
Then there exists a subsequence of $\{\alpha_n\}$, which is denoted by $\{\alpha_n\}$ again, 
such that 
\begin{equation}
	v(s, \alp_n) \to z(s, \sigma) \quad \mbox{in} \ C[0, \hat{S}] 
	\quad \mbox{as} \ n \to \infty.
\label{eq5.16}
\end{equation}
\end{lemma}

\begin{proof}
For simplicity, we denote by $v_n(s) = v(s, \alpha_n)$. 
From (\ref{eq5.15}), we have $s_{\gamma_n} > \hat{S}$ for sufficiently large $n$. 
From (\ref{eq5.11}) and Lemma \ref{lem5.2} (ii) we have 
\begin{equation}
	\tau \leq v_n(s) \leq \sigma 
	\quad \mbox{for} \ 0 \leq s \leq \hat{S}
\label{eq5.17}
\end{equation}
and 
\begin{equation}
	\left|\frac{v_n'(s)}{s}\right| \leq C_1 
	\quad \mbox{and} \quad 
	|v_n'(s)| \leq C_1\hat{S}
	\quad \mbox{for} \ 0 < s \leq \hat{S} 
\label{eq5.18}
\end{equation}
for sufficiently large $n$, 
where $C_1 > 0$ is a constant independent of $n$.
By Lemma \ref{lem5.2} (i) and (\ref{eq5.18}), we have 
\begin{equation}
	\sup_{s \in [0, \hat{S}]}\left|\frac{f'(u(r, \beta_n)))F(u(r, \beta_n)))-q_0}
	{g(v_n(s))G(v_n(s))}v_n'(s)^2\right| \to 0 
	\quad \mbox{as} \ n \to \infty,
\label{eq5.19}
\end{equation}
where $\beta_n = \beta(\alpha_n)$. 
Integrating (\ref{eq5.5}) on $[0, s]$ with $0 < s \leq \hat{S}$, 
we obtain  
\begin{equation}
	v_n'(s) = -\frac{1}{s^{N-1}}\int^{s}_0 t^{N-1}\left(
	g(v_n(t)) + \frac{f'(u(r, \beta_n))F(u(r, \beta_n))-q_0}{g(v_n(t)G(v_n(t))}
	v_n'(t)^2 \right)ds.
\label{eq5.20}
\end{equation}
Note that $v_n(0) = \sigma$ and $v_n'(0) = 0$. 
Making use of L'Hospital's rule in (\ref{eq5.20}), we have 
$$
	v_n''(0) = \lim_{s \to 0}\frac{v_n'(s)}{s} 
	= -\frac{g(\sigma)}{N}.
$$
Then $v_n \in C^2[0, \hat{S}]$. 
Using (\ref{eq5.17}), (\ref{eq5.18}) and (\ref{eq5.19}) in (\ref{eq5.5}), we see that 
$v_n''(\xi)$ is uniformly bounded on $s \in [0, \hat{S}]$ for sufficiently large $n$.  
By the Ascoli-Arzel\'{a} theorem, 
there exist $\bar{v} \in C^1[0, \hat{S}]$ and 
a subsequence, which is denoted by $\{v_n\}$ again, 
such that $v_n \to \bar{v}$ in $C^1[0, \hat{S}]$ as $n \to \infty$. 
Letting $n \to \infty$ in (\ref{eq5.20}), we obtain 
$$
	\bar{v}'(s) = -\frac{1}{s^{N-1}}\int^{s}_0 t^{N-1}
	g(\bar{v}(t)) dt \quad 
	\mbox{for} \ 0 < s \leq \hat{S},
$$
which implies that $\bar{v}'(0) = 0$. 
Since $\bar{v}(0) = \sigma$, $\bar{v}$ solves (\ref{eq5.1}), 
that is $\bar{v}(s) \equiv z(s, \sigma)$. 
Thus we obtain (\ref{eq5.16}).
\end{proof}

\begin{proof}[Proof of Proposition $\ref{prp5.1}$] 
Let $\{\alp_n\}$ be a sequence satisfying $0 < \alp_n \to 0$ as $n \to \infty$. 
Denote by $\gamma_n = \gamma(\alpha_n)$, $s_{\gamma_n} = s_{\gamma(\alpha_n)}$.
Let 
\begin{equation}
	S_{\infty} = \liminf_{n \to \infty}s_{\gamma_n}.
\label{eq5.21}
\end{equation}
Then $S_{\infty} > 0$ by Lemma \ref{lem5.2} (iii). 
We will show that $S_{\infty} > S$. 
Assume by contradiction that $S_{\infty} \leq S$. 
From (\ref{eq5.21}) 
there exists a subsequences, which we denote by $\{\alpha_n\}$ and $\{\gamma_n\}$ again, 
such that $\lim_{n \to \infty}s_{\gamma_n} = S_{\infty}$. 
Recall that $v(s_{\gamma_n}, \alpha_n) = \tau < z(S, \sigma)$ for all $n \in \N$. 
For $\delta > 0$ to be chosen later, we have
$$
	v(s_{\gamma_n}-2\delta, \alpha_n) = v(s_{\gamma_n}, \alpha_n) 
	-\int^{s_{\gamma_n}}_{s_{\gamma_n}-2\delta} v_s(\xi, \alpha_n)d\xi
	\leq \tau + 
	\int^{s_{\gamma_n}}_{s_{\gamma_n}-2\delta} |v_s(\xi, \alpha_n)|d\xi.
$$
By Lemma \ref{lem5.2} (ii), 
$\sup_{s \in [0, s_{\gamma_n}]}|v_s(s, \alpha_n)|$ is uniformly bounded for 
sufficiently large $n$. 
Then we can choose $\delta > 0$ which is independent of $n$ such that 
$$
	v(s_{\gamma_n}-2\delta, \alpha_n) < z(S, \sigma)
$$
for sufficiently large $n$. 
Recall that $s_{\gamma_n} \to S_{\infty}$ as $n \to \infty$. 
Then we have 
$$
	s_{\gamma_n}-2\delta < S_{\infty}-\delta < s_{\gamma_n}
$$
for sufficiently large $n$. 
Since $v(s, \alpha_n)$ is decreasing in $s > 0$, we have 
\begin{equation}
	z(S, \sigma) > v(s_{\gamma_n}-2\delta, \alpha_n) > v(S_{\infty}-\delta, \alpha_n)
\label{eq5.22}
\end{equation}
for sufficiently large $n$. 
Applying Lemma \ref{lem5.4} with $\hat{S} = S_{\infty} -\delta$, we see that 
there exists a subsequence of $\{\alpha_n\}$, which is denoted by $\{\alpha_n\}$ again, 
such that 
\begin{equation}
	v(s, \alp_n) \to z(s, \sigma) \quad \mbox{in} \ C[0, S_{\infty}-\delta] 
	\quad \mbox{as} \ n \to \infty.
\label{eq5.23}
\end{equation}
Note that $z(s, \sigma)$ is monotone decreasing in $s > 0$ and 
$S_{\infty} -\delta < S_{\infty} \leq S$. 
Then, from (\ref{eq5.23}) we obtain 
$$
	\lim_{n \to \infty}v(S_{\infty}-\delta, \alp_n) 
	= z(S_{\infty}-\delta, \sigma) > z(S, \sigma).
$$
On the other hand, from (\ref{eq5.22}) we obtain 
$$
	\lim_{n \to \infty}v(S_{\infty}-\delta, \alpha_n) \leq z(S, \sigma),
$$
which is a contradiction. 
Thus we obtain $S_{\infty} > S$. 

Now, applying Lemma \ref{lem5.4} with $\hat{S} = S$, we obtain 
\begin{equation}
	v(s, \alpha_{n_k}) \to z(s, \sigma) \quad \mbox{in} \ 
	C[0, S] \quad \mbox{as} \ n \to \infty
\label{eq5.24}
\end{equation}
for some subsequence $\{\alpha_{n_k}\}$ of $\{\alpha_n\}$. 
Therefore, for any sequence $\{\alpha_n\}$, 
there exists a subsequence $\{\alpha_{n_k}\}$ such that  
(\ref{eq5.24}) holds. 
This implies that (\ref{eq5.6}) holds.
\end{proof}

We recall some properties of solutions to (5.1) from \cite{Wan}. 

\begin{lemma}\label{lem5.5}
%{\bf Lemma 5.5.} \it 
Let $N \geq 3$, and let $z(s, \sigma)$ be a solution of $(\ref{eq5.1})$ with $g(u) = u^p$. 
Then, for any $\sigma > 0$, $z_0(\sigma) < \infty$ if $1 < p < p_S$ and 
$z_0(\sigma) = \infty$ if $p \geq p_S$. 
Furthermore, for any $\sigma_1 > \sigma_2 > 0$, the following hold. 
\begin{enumerate}
\item[{\rm (i)}] 
If $1 < p < p_S$ then $z(s, \sigma_1)-z(s, \sigma_2)$ has 
at least one zero in $(0, \min\{z_0(\sigma_1), z_0(\sigma_2)\})$. 

\item[{\rm (ii)}] 
If $p = p_S$ then $z(s, \sigma_1)-z(s, \sigma_2)$ has 
exactly one zero for $s > 0$.

\item[{\rm (iii)}] 
If $p_S < p < p_{JL}$ then $z(s, \sigma_1)-z(s, \sigma_2)$ has 
infinitely many zeros for $s > 0$.
\end{enumerate}
\end{lemma}

\begin{proof}[Proof of Theorem $\ref{thm1.5}$] 
Set $p = q_0/(q_0-1)$. 
Since $q_0 > q_{JL}$, we have $1 < p < p_{JL}$.
In this case, Lemma \ref{lem5.5} implies that, 
for any $\sigma_1 > \sigma_2 > 0$, $z(s, \sigma_1)-z(s, \sigma_2)$ has 
at least one zero for $s > 0$.
Take $S > 0$ such that $z(s, 1)-z(s, 1/2)$ has at least one zero in $(0, S)$.

First, we define $\beta(\alpha)$ by (\ref{eq5.2}) with $\sigma = 1$, that is, $\beta(\alpha) = \alpha$, 
and define $v_1(s, \alpha)$ by  
\begin{equation}
	v_1(s, \alpha) = G^{-1}(\lam^{-2}F(u(r, \alpha))) 
	\quad \mbox{with} \ 
	s = \frac{r}{\lam} \quad \mbox{and} \quad 
	\lam = \sqrt{\frac{F(\alpha)}{G(1)}}.
\label{eq5.25}
\end{equation}
Then, by Proposition \ref{prp5.1}, we have 
\begin{equation}
	v_1(s, \alp) \to z(s, 1) \quad \mbox{in} \ C[0, S] 
	\quad \mbox{as} \ \alpha \to 0.
\label{eq5.26}
\end{equation}
Next, define $\beta(\alpha)$ by (\ref{eq5.2}) with $\sigma = 1/2$, and define $v_2(s, \alpha)$ by 
\begin{equation}
	v_2(s, \alpha) = G^{-1}(\lam^{-2}F(u(r, \beta(\alpha)))) 
	\quad \mbox{with} \ 
	s = \frac{r}{\lam} \quad \mbox{and} \quad 
	\lam = \sqrt{\frac{F(\alpha)}{G(1)}}.
\label{eq5.27}
\end{equation}
By Proposition \ref{prp5.1}, we obtain 
\begin{equation}
	v_2(s, \alp) \to z(s, 1/2) \quad \mbox{in} \ C[0, S] 
	\quad \mbox{as} \ \alpha \to 0.
\label{eq5.28}
\end{equation}
Recall that $z(s, 1)-z(s, 1/2)$ has at least one zero in $(0, S)$. 
Then, from (\ref{eq5.26}) and (\ref{eq5.28}), 
$v_1(s, \alpha)-v_2(s, \alpha)$ has at least one zero in $(0, S)$ 
for sufficiently small $\alpha$. 
It follows from (\ref{eq5.25}) and (\ref{eq5.27}) that 
$u(r, \alpha)-u(r, \beta(\alpha))$ has at least one zero 
for sufficiently small $\alpha$. 
Lemma \ref{lem2.2} (ii) implies that 
$u(r, \alpha)$ is unstable for sufficiently small $\alpha$, 
and hence (I) holds in Theorem \ref{thm1.1}. 
Thus $u(r, \alpha)$ is unstable for all $\alpha > 0$. 
\end{proof}

\subsection{Proof of Theorem $1.6$} 

In this subsection, we assume that $N \geq 11$ and 
(f1) holds with $q_{\infty} > q_{JL} > 1$.
We prove Theorem \ref{thm1.6} by a similar argument to that in the proof of Theorem \ref{thm1.5}, 
replacing $\alpha \to 0$ by $\alpha \to \infty$. 
Define $g(u) = u^p$ with $p = q_{\infty}/(q_{\infty}-1)$, 
and consider the initial value problem (\ref{eq5.1}). 

Let $\sigma > 0$ be fixed. 
For $\alpha > 0$, define $\beta = \beta(\alpha) > 0$ by (\ref{eq5.2}).
For a solution $u(r, \alpha)$ of (\ref{eq1.4}), define $v(s, \alpha)$ by (\ref{eq5.3}).
Then, by Lemma \ref{lem3.1}(i), $v(s) = v(s, \alpha)$ satisfies 
$$
	\left\{
	\begin{array}{c}
	\dsp
	v''+\frac{N-1}{s}v'+g(v) + 
	\frac{f'(u(r, \beta))F(u(r, \beta))-q_{\infty}}{g(v(s))G(v(s))}v'(s)^2 = 0 \quad \mbox{for} \  s > 0,
	\\[3ex]
	v(0) = \sigma \quad \mbox{and} \quad v'(0) = 0.
	\end{array}
	\right.
$$
Since $F(u) \to 0$ as $u \to \infty$,
we have $\lam = \lam(\alpha) \to 0$ as $\alpha \to \infty$. 
We will show the following proposition.

\begin{proposition}\label{prp5.2}
%{\bf Proposition 5.2.} \it 
For any $S \in (0, s_0(\sigma))$, one has 
$$
	v(s, \alp) \to z(s, \sigma) \quad \mbox{in} \ C[0, S] 
	\quad \mbox{as} \ \alpha \to \infty,
$$
where $z(s, \sigma)$ is the solution of $(\ref{eq5.1})$. 
\end{proposition}

Take any $\tilde{S} \in (S, s_0(\sigma))$, and let  $\tau = z(\tilde{S}, \sigma)$. 
Define $\gamma = \gamma(\alpha) > 0$ by (\ref{eq5.7}). 
Then there exists $r_{\gamma} > 0$ satisfying (\ref{eq5.8}) and we have (\ref{eq5.9}).
Define $s_{\gamma} = \lam^{-1}r_{\gamma}$. 
Then we obtain (\ref{eq5.11}).

We will show that 
\begin{equation}
	\sup_{s \in [0, s_{\gamma}]}
	\frac{|f'(u(r, \beta))F(u(r, \beta))-q_{\infty}|}{g(v(s, \alpha))G(v(s, \alpha))} \to 0 
	\quad \mbox{as} \ \alpha \to \infty.
\label{eq5.29}
\end{equation}
In fact, from (\ref{eq5.7}) we see that $\gamma(\alpha) \to \infty$ as $\alpha \to \infty$.
From (\ref{eq5.9}) we have, for any $r \in [0, r_{\gamma}]$,
$$
	u(r, \beta) \geq \gamma(\alpha) \to \infty
	\quad \mbox{as} \ \alpha \to 0.
$$ 
The interval $0 \leq s \leq s_{\gamma}$ corresponds to $0 \leq r \leq r_{\gamma}$. 
Then, from (\ref{eq1.6}) we obtain 
$$
	\sup_{s \in [0, s_{\gamma}]}|f'(u(r, \beta))F(u(r, \beta))-q_{\infty}| \to 0 
	\quad \mbox{as} \ \alpha \to \infty.
$$
From (\ref{eq5.11}) we have 
$$
	g(v(s, \alpha))G(v(s, \alpha)) = \frac{v(s, \alpha)}{p-1} \geq \frac{\tau}{p-1} > 0 
	\quad \mbox{for} \ 0 \leq s \leq s_{\gamma}. 
$$
Thus we obtain (\ref{eq5.29}).

Since we can prove Proposition \ref{prp5.2} by a similar argument to that 
in the proof Proposition \ref{prp5.1}, 
we omit the proof of Proposition \ref{prp5.2}.

\begin{proof}[Proof of Theorem $\ref{thm1.6}$] 
Set $p = q_{\infty}/(q_{\infty}-1)$. 
Since $q_{\infty} > q_{JL}$, we have $1 < p < p_{JL}$.
In this case, Lemma \ref{eq5.5} implies that, 
for any $\sigma_1 > \sigma_2 > 0$, $z(s, \sigma_1)-z(s, \sigma_2)$ has 
at least one zero for $s > 0$.
By a similar argument to that in the proof of Theorem \ref{thm1.5}, 
we conclude that $u(r, \alpha)$ is unstable for sufficiently large $\alpha$. 
\end{proof}

%%%%%%%%%%%%%%%%%%%%%%%%%%%%%%%

\section{Proof of Theorems \ref{thm1.8} and \ref{thm1.9}}

First we will show the following lemma.

\begin{lemma}\label{lem6.1}
%{\bf Lemma 6.1. } \it 
Let $u(r, \alpha)$ be a solution of $(1.4)$ with $\alpha > 0$. 
Assume that $r_0(\alpha) = \infty$, that is, $u(r, \alpha) > 0$ for all $r \geq 0$.
Then $F(u(r, \alpha)) \geq r^2/(2N)$ for $r \geq 0$. 
\end{lemma}

\begin{proof}
Integrating the equation in  (\ref{eq1.4}) on $[0, r]$, we obtain 
\begin{equation}
	-u'(r, \alpha) = \frac{1}{r^{N-1}}\int^r_0 s^{N-1}f(u(s, \alpha))ds < 0, 
	\quad r > 0.
\label{eq6.1}
\end{equation}
By the monotonicity of $u$ and $f$, we obtain 
$$
	-r^{N-1}u'(r, \alpha) = \int^r_0 s^{N-1}f(u(s, \alpha))ds  
	\geq f(u(r, \alpha))\int^r_0 s^{N-1}ds = \frac{r^N}{N}f(u(r, \alpha)), 
$$
which implies that 
$$
	-\frac{u'(r, \alpha))}{f(u(r, \alpha))} \geq \frac{r}{N}.
$$
Integrating the above on $[0, r]$, we obtain 
$$
	F(u(r, \alpha)) > F(u(r, \alpha))-F(\alpha) \geq \frac{r^2}{2N}.
$$ 
\end{proof}

\begin{proof}[Proof of Theorem $\ref{thm1.8}$]
Since (II) holds in Theorem \ref{thm1.1}, 
$u(r, \alpha) > 0$ for all $r \geq 0$ and $u(r, \alpha)$ is monotone increasing in $\alpha > 0$. 
Then, by Lemma \ref{lem6.1}, there exists $u^*(r)$ such that 
$$
	\lim_{\alpha \to \infty}u(r, \alpha) = u^*(r) \leq F^{-1}(r^2/2N) 
	\quad \mbox{for} \ r > 0.
$$
Since $u(r, \alpha) > 0$ for $r \geq 0$, $u^*(r) > 0$ for all $r > 0$.
For any $\alpha > 0$, there exists $r_{\alpha} > 0$ such that 
$u(r, \alpha+1) \geq \alpha$ for $0 < r \leq r_{\alpha}$.
Since $u^*(r) > u(r, \alpha+1)$, we have $u^*(r) > \alpha$ for $0 < r \leq r_{\alpha}$, 
which implies that $u^*(r) \to \infty$ as $r \to 0$.

We will show that $u^*(r)$ solves (\ref{eq1.12}) and satisfies    
\begin{equation}
	\int^r_0 s^{N-1}f(u^*(s))ds < \infty
	\quad \mbox{for} \ r > 0.
\label{eq6.2}
\end{equation}
In fact, integrating (\ref{eq6.1}) on $[r, r_0]$, we obtain 
$$
	\begin{array}{rcl}
	-u(r_0, \alpha) + u(r, \alpha) & = & \dsp
	 \int^{r_0}_r s^{1-N}\int^s_0 t^{N-1}f(u(t, \alpha))dtds.
	\\[2ex]
	 & = & \dsp
	 -\int^{r_0}_r\frac{d}{ds}\left(\frac{s^{2-N}-r_0^{2-N}}{N-2}\right)
	 \int^s_0 t^{N-1}f(u(t, \alpha))dtds.
	\end{array}
$$
Integrating by parts, we obtain 
$$
	\begin{array}{rcl}
	-u(r_0, \alpha) + u(r, \alpha) & = &
	\dsp
	\frac{r^{2-N}-r_0^{2-N}}{N-2}\int^{r}_0 s^{N-1}f(u(s, \alpha))ds 
	\\[2ex]
	& & \dsp
	+ \int^{r_0}_r 
	\frac{s^{2-N}-r_0^{2-N}}{N-2}s^{N-1}f(u(s, \alpha))ds.
	\end{array}
$$
Letting $\alpha \to \infty$, by the monotone convergence theorem, we obtain 
$$
	-u^*(r_0) + u^*(r) = 
	\frac{r^{2-N}-r_0^{2-N}}{N-2}\int^{r}_0 s^{N-1}f(u^*(s))ds 
	+ \int^{r_0}_r \frac{s^{2-N}-r_0^{2-N}}{N-2}s^{N-1}f(u^*(s))ds.
$$
This implies that $u^*(r)$ satisfies (\ref{eq1.12}) and (\ref{eq6.2}) for $0 < r \leq r_0$. 
Since $r_0 > 0$ is arbitrary, $u^*(r)$ solves (\ref{eq1.12}) for all $r > 0$.  

Take any $\phi \in C^1_0(\R^N)$. 
Since $u(r, \alpha)$ is stable for $\alpha > 0$, we obtain 
$$
	\int_{\R^N}|\nabla \phi|^2 dx 
	\geq 
	\int_{\R^N}f'(u(|x|, \alpha))\phi^2 dx. 
$$
Letting $\alpha \to \infty$, by the monotone convergence theorem, we obtain 
$$
	\int_{\R^N}|\nabla \phi|^2 dx 
	\geq 
	\int_{\R^N}f'(u^*(|x|))\phi^2 dx. 
$$
Thus $f'(u^*(|\cdot|)) \in L^1_{\rm loc}(\R^N)$ and (\ref{eq1.14}) holds. 
\end{proof}

To prove Theorem \ref{thm1.9}, we need a series of some lemmas.

\begin{lemma}\label{lem6.2}
%{\bf Lemma 6.2. } \it 
Assume that $(\ref{eq1.15})$ holds. 
Then, there exists a constant $C > 0$ such that 
\begin{equation}
	F(u) \leq Cu^{-\frac{4}{N-2}} 
	\quad \mbox{for sufficiently large $u$.}
\label{eq6.3}
\end{equation}
\end{lemma}

\begin{proof}
From (\ref{eq1.15}) there exists $u_0 \geq 0$ such that 
$f'(u)F(u) \leq q_S$ for $u \geq u_0$. 
Then it follows that 
$$
	\frac{d}{du}\left(f(u)F(u)^{q_S}\right) 
	= F(u)^{q_S-1}\left(f'(u)F(u)-q_S\right) 
	\leq 0 \quad \mbox{for} \ u \geq u_0.
$$
Thus $f(u)F(u)^{q_S}$ is nonincreasing for $u \geq u_0$, 
and hence we obtain $f(u)F(u)^{q_S} \leq C$ for $u \geq u_0$  
with $C = f(u_0)F(u_0)^{q_S} > 0$.  
This implies that 
$$
	-F'(u)F(u)^{-q_S} \geq \frac{1}{C} \quad \mbox{for} \ u \geq u_0.
$$
Integrating the above on $[u_0, u]$, we obtain 
$$
	\frac{1}{q_S-1}
	\left(F(u)^{1-q_S} - F(u_0)^{1-q_S}\right)
	= -\int^u_{u_0} F'(u)F(u)^{-q_S}du \geq \frac{1}{C}(u-u_0).
$$
Then we obtain (\ref{eq6.3}) for some constant $C > 0$. 
\end{proof}

\begin{lemma}\label{lem6.3}
%{\bf Lemma 6.3. } \it 
Assume that $(\ref{eq1.15})$ holds. 
Let $u^*(r)$ be a singular radial solution of $(\ref{eq1.1})$. 
Then 
\begin{equation}
	u^*(r) = O(r^{-\frac{N-2}{2}}) \quad 
	\mbox{as} \ r \to 0.
\label{eq6.4}
\end{equation}
\end{lemma}

\begin{proof}
By Lemma \ref{lem2.3} in \cite{MiNab}, we obtain 
$$
	F(u^*(r)) \geq \frac{r^2}{2N} 
	\quad \mbox{for}  \ 0 < r \leq r_0
$$
with some $r_0 > 0$. 
By Lemma \ref{lem6.2}, there exists a constant $C > 0$ such that 
$$
	C u^*(r)^{-\frac{4}{N-2}} \geq 
	F(u^*(r)) \geq \frac{r^2}{2N}
	\quad \mbox{for}  \ 0 < r \leq r_1
$$
with some $r_1 \leq r_0$. 
Thus we obtain (\ref{eq6.4}). 
\end{proof}

Let $u^*(r)$ be the singular radial solution of (\ref{eq1.1}) satisfying 
(\ref{eq1.13}) and (\ref{eq1.14}), and let us consider the equation 
\begin{equation}
	\phi'' + \frac{N-1}{r}\phi' + f'(u^*(r))\phi = 0  \quad 
	\mbox{for} \ r > 0.
\label{eq6.5}
\end{equation}

\begin{lemma}\label{lem6.4}
%{\bf Lemma 6.4. } \it 
Let $u^*(r)$ be the singular radial solution of $(\ref{eq1.1})$ 
satisfying $(\ref{eq1.13})$ and $(\ref{eq1.14})$.
Then, for any $r_0 > 0$, there exists a solution $\phi$ of $(\ref{eq6.5})$ satisfying 
$\phi(r) > 0$ for $0 < r \leq r_0$. 
\end{lemma}

\begin{proof}
First we will show that any solution $\phi$ of (\ref{eq6.5}) has at most one zero in $(0, \infty)$. 
Assume by contradiction that (\ref{eq6.5}) has a solution $\hat{\phi}$ 
which has at least two zeros $r_1 < r_2$ in $(0, \infty)$. 
Let us consider the initial value problem 
\begin{equation}
	\left\{
	\begin{array}{c}
	\dsp 
	\phi'' + \frac{N-1}{r}\phi' + \lam f'(u^*(r))\phi = 0  \quad 
	\mbox{for} \ r > 0, 
	\\[2ex]
	\phi(r_1) = 0 \quad \mbox{and} \quad \phi'(r_1) = 1,
	\end{array}
	\right.
\label{eq6.6}
\end{equation}
where $\lam \in (0, 1]$ is a parameter, 
and we denote by $\phi_{\lam}$ a unique solution of (\ref{eq6.6}). 
In the case $\lam = 1$, $\phi_{\lam}(r)$ is the constant multiple of $\hat{\phi}$, 
and $\phi_{\lam}$ satisfies $\phi_{\lam}(r_2) = 0$. 
By the continuous dependence of solutions on the parameter $\lam$, 
$\phi_{\lam}$ has at least one zero in $(r_1, r_2 + 1)$ 
if $\lam \in (0, 1)$ is sufficiently close to $1$. 
We denote by $r_{\lam}$ the first zero of $\phi_{\lam}(r)$ for $r > r_1$, that is, 
$\phi_{\lam}(r) > 0$ for $r_1 < r < r_{\lam}$ and $\phi_{\lam}(r_{\lam}) = 0$.
Extend $\phi_{\lam}$ as $\phi_{\lam}(r) \equiv 0$ for $0 < r < r_1$ and $r > r_{\lam}$. 
Then $\phi_{\lam}(|x|)$ with $x \in \R^N$ satisfies $\phi_{\lam}(|\cdot|) \in H^1(\R^N)$ and 
$$
	\Delta \phi_{\lam} + \lam f'(u^*(|x|))\phi_{\lam} = 0
	\quad \mbox{in} \ r_1 < |x| < r_{\lam}.
$$
Multiplying the both sides of the above by $\phi_{\lam}$, and integrating it on $\R^N$, 
we obtain 
$$
	\int_{\R^N}|\nabla \phi_{\lam}|^2 dx = 
	\lam \int_{\R^N} f'(u^*(|x|))\phi_{\lam}^2 dx
	< \int_{\R^N}f'(u^*(|x|))\phi_{\lam}^2 dx.
$$
Since $C^1_0(\R^N)$ is dense in $H^1(\R^N)$, 
there exists $\psi \in C^1_0(\R^N)$ satisfying 
$$
	\int_{\R^N}(|\nabla \psi|^2 - f'(u^*(|x|))\psi^2) dx < 0, 
$$
which contradicts (\ref{eq1.14}). 
Thus, any solution of (\ref{eq6.5}) has at most one zero in $(0, \infty)$.

Let us consider a solution $\phi$ of (\ref{eq6.5}) satisfying 
$$
	\phi(r_0 + 1) = 0 \quad \mbox{and} \quad \phi'(r_0+1) = -1.
$$
Since $\phi(r)$ as at most one zero in $(0, \infty)$, 
the solution $\phi$ satisfies $\phi(r) > 0$ for $0 < r \leq r_0$.
\end{proof}

We recall the comparison lemma in \cite[Proposition 4.1]{MiNaa}.

\begin{lemma}\label{lem6.5}
%{\bf Lemma 6.5.} \it 
Let us consider two differential equations 
\begin{equation}
	u'' + \frac{N-1}{r}u' + V(r)u = 0, \quad 0 < r < R,
\label{eq6.7}
\end{equation}
\begin{equation}
	v'' + \frac{N-1}{r}v' + \widehat{V}(r)v = 0, \quad 0 < r < R,
\label{eq6.8}
\end{equation}
where $V, \widehat{V} \in C(0, R]$ satisfy 
$V(r) < \widehat{V}(r)$ for $0 < r \leq R$. 
Assume that $(\ref{eq6.7})$ has a solution $u_0$ 
satisfying $u_0(r) > 0$ for $0 < r < R$, $u_0(R) = 0$ and 
$$
	\int_0\frac{dr}{r^{N-1}u_0(r)^2} = \infty.
$$
Then any solution $v$ of $(\ref{eq6.8})$ has at least one zero in $(0, R)$.
\end{lemma}

\begin{proof}[Proof of Theorem $\ref{thm1.9}$] 
Take any $\alpha > 0$. 
First we will show that 
\begin{equation}
	u^*(r) > u(r, \alpha) \quad \mbox{for all} \ 0 < r < r_0(\alpha) \leq \infty.
\label{eq6.9}
\end{equation}
Define $w(r) = u^*(r) - u(r, \alpha)$ for $0 < r < r_0(\alpha)$. 
Assume by contradiction that there exists $r_1 \in (0, r_0(\alpha))$ such that 
\begin{equation}
	w(r) > 0 \quad \mbox{for} \ 0 < r < r_1 \quad 
	\mbox{and} \quad w(r_1) = 0.
\label{eq6.10}
\end{equation}
Then $w$ satisfies 
\begin{equation}
	w'' + \frac{N-1}{r}w' + V(r) w = 0
	\quad \mbox{for} \ 0 < r < r_1, 
\label{eq6.11}
\end{equation}
where 
$$
	V(r) = \frac{f(u^*(r))-f(u(r, \alpha))}{u^*(r)-u(r, \alpha)} 
	\quad \mbox{for} \ 0 < r < r_1.
$$
Since $f'(u)$ is strictly increasing for $u > 0$ and 
$u^*(r) > u(r, \alpha)  > 0$ for $0 < r < r_1$, we have 
$$
	V(r) < f'(u^*(r)) \quad \mbox{for} \ 0 < r < r_1.
$$
Since $w(r) < u^*(r)$ for $0 < r < r_1$, from Lemma \ref{lem6.3} we have 
$$
	\int_0 \frac{dr}{r^{N-1}w(r)^2} > \int_0 \frac{dr}{r^{N-1}u^*(r)^2}
	= \infty.
$$
By Lemma \ref{lem6.5}, any solution $\phi$ of (\ref{eq6.5}) has at least one zero in $(0, r_1)$, 
which contradicts Lemma \ref{lem6.4}. 
Thus (\ref{eq6.9}) holds. 

Next we will show that, for any $\alpha > \beta > 0$,  
\begin{equation}
	u(r, \alpha) > u(r, \beta) \quad \mbox{for all} \ 0 \leq r < r_0(\beta) \leq \infty.
\label{eq6.12}
\end{equation}
Define $w(r) = u(r, \alpha) - u(r, \beta)$ for $0 \leq r < r_0(\beta)$. 
Assume by contradiction that there exists $r_1 \in (0, r_0(\beta))$ such that 
(\ref{eq6.10}) holds. 
We see that $w$ satisfies (\ref{eq6.11}) with 
$$
	V(r) = \frac{f(u(r, \alpha))-f(u(r, \beta))}
	{u(r, \alpha)-u(r, \beta)}
	\quad \mbox{for} \ 0 < r < r_1.
$$
Since $u^*(r) > u(r, \alpha)  > u(r, \beta)$ for $0 < r < r_1$, we have 
$$
	V(r) < f'(u(r, \alpha)) < f'(u^*(r)) \quad \mbox{for} \ 0 < r < r_1.
$$
Since $w(0) = \alpha-\beta > 0$, it is clear that $w$ satisfies 
$$
	\int_0 \frac{dr}{r^{N-1}w(r)^2} = \infty.
$$
Lemma \ref{lem6.5} implies that any solution $\phi$ of (\ref{eq6.5}) has at least 
one zero in $(0, r_1)$, which contradicts Lemma \ref{lem6.4}. 
Thus (\ref{eq6.12}) holds, and this implies that $r_0(\beta) = \infty$. 
In fact, if $r_0(\beta) < \infty$, then $u(r, \alpha)-u(r, \beta)$ has 
at least one zero in $(0, r_0(\beta))$ by Lemma \ref{lem2.1}, 
which contradicts (\ref{eq6.12}). 
Thus, for any $\alpha > \beta > 0$, we obtain 
$$
	u(r, \alpha) > u(r, \beta) > 0 \quad \mbox{for all} \ r \geq 0.
$$
By Lemma \ref{lem2.2} (i), $u(r, \beta)$ is stable for any $\beta > 0$, 
and hence (II) holds in Theorem~\ref{thm1.1}.
\end{proof}

%%%%%%%%%%%%%%%%%%%%%%%%%%%%

\section{Two examples}

In this section we give two examples of $f$ satisfying the hypotheses in 
Corollaries \ref{cor1.4} and \ref{cor1.7}.

\begin{example}\label{exm7.1}
%{\bf Example 7.1.} 
Assume that constants $q_1$ and $q_2$ satisfy 
$1 < q_1 < q_2$ and (\ref{eq1.8}). 
In (\ref{eq1.1}), let 
$$
	f(u) = u^{p_1} + u^{p_2}, \quad \mbox{where} \ 
	p_1 = \frac{q_1}{q_1-1} \quad \mbox{and} \quad p_2 = \frac{q_2}{q_2-1}.
$$
Then $f$ satisfies (\ref{eq1.10}) in Corollary \ref{cor1.4}. 
In fact, by a direct calculation, we have 
\begin{equation}
	\begin{array}{cl}
	& \dsp
	\frac{f'(u)^2}{f(u)f''(u)} 
	\\[2ex]
	= & \dsp
	\frac{p_1^2 u^{2p_1-2} + 2p_1p_2u^{p_1+p_2-2} + p_2^2 u^{2p_2-2}}
	{p_1(p_1-1)u^{2p_1-2} + (p_1(p_1-1)+ p_2(p_2-1))u^{p_1+p_2-2} + p_2(p_2-1)u^{2p_2-2}}
	\\[4ex]
	= & \dsp
	\frac{\left(\frac{p_1^2}{p_2^2}u^{2p_1-2p_2} + \frac{2p_1}{p_2}u^{p_1-p_2} + 1\right) p_2^2 }
	{\left(\frac{p_1(p_1-1)}{p_2(p_2-1)}u^{2p_1-2p_2} + (\frac{p_1(p_1-1)}{p_2(p_2-1)}+ 1)u^{p_1-p_2} 
	+ 1\right)p_2(p_2-1)}.
	\end{array}
	\label{eq7.1}
\end{equation}
Since $p_1 > p_2 > 1$, we have 
$$
	\frac{p_1(p_1-1)}{p_2(p_2-1)} > \frac{p_1^2}{p_2^2} 
$$
and 
$$
	\frac{p_1(p_1-1)}{p_2(p_2-1)}+ 1 = \frac{p_1^2 + p_2^2 -p_1-p_2}{p_2(p_2-1)}
	> \frac{2p_1p_2-2p_1}{p_2(p_2-1)} = \frac{2p_1}{p_2}.
$$
Thus we obtain 
$$
	\frac{\left(\frac{p_1^2}{p_2^2}u^{2p_1-2p_2} + \frac{2p_1}{p_2}u^{p_1-p_2} + 1\right)} 
	{\left(\frac{p_1(p_1-1)}{p_2(p_2-1)}u^{2p_1-2p_2} + (\frac{p_1(p_1-1)}{p_2(p_2-1)}+ 1)u^{p_1-p_2} 
	+ 1\right)} \leq 1 \quad \mbox{for} \ u \geq 0.
$$
Then it follows from (\ref{eq7.1}) that 
$$
	\frac{f'(u)^2}{f(u)f''(u)} \leq \frac{p_2}{p_2-1} = q_2 \quad 
	\mbox{for} \ u > 0.
$$
By a similar argument, we obtain 
$$
	\begin{array}{rcl}
	\dsp
	\frac{f'(u)^2}{f(u)f''(u)} & = & \dsp
	\frac{p_1^2 \left(u^{2p_1-2p_2} + \frac{2p_2}{p_1}u^{p_1-p_2} + \frac{p_2^2}{p_1^2}\right)}
	{p_1(p_1-1)\left(u^{2p_1-2p_2} + (1+ \frac{p_2(p_2-1)}{p_1(p_1-1)})u^{p_1-p_2} + \frac{p_2(p_2-1)}{p_1(p_1-1)}\right)}.
	\\[4ex]
	& \geq & \dsp 
	\frac{p_1}{p_1-1} = q_1 
	\quad \mbox{for} \ u > 0.
	\end{array}
$$
Thus (\ref{eq1.10}) holds.
By Corollary \ref{cor1.4}, we see that (II) holds in Theorem \ref{thm1.1}.
\end{example}

\begin{example}\label{exm7.2}
%{\bf Example 7.2. } 
In (\ref{eq1.1}) let 
$$
	f(u) = \frac{u^{p_1}}{(1+u)^{p_2}} \quad \mbox{with} \ 
	p_1 > p_{JL} \quad \mbox{and} \quad 1 \leq p_1-p_2 < p_{JL}.
$$
We will show that $f'(u) > 0$ and $f''(u) > 0$ for $u > 0$, and  that 
(f1) and (f2) hold with $q_0 < q_{JL} < q_{\infty}$. 
By a direct calculation, we have 
$$
	f'(u) = \frac{p_1u^{p_1-1}+(p_1-p_2)u^{p_1}}{(1+u)^{p_2+1}} 
	\quad \mbox{and} 
$$
$$
	f''(u) = \frac{((p_1-p_2)(p_1-p_2-1)u^2 + 2p_1(p_1-p_2-1)u + p_1(p_1-1))u^{p_1-2}}
	{(1+u)^{p_2+2}}. 
$$
Recall that $p_1 > p_2+1$. Then we obtain $f'(u) > 0$ and $f''(u) > 0$ for $u > 0$. 

It follows that  
$$
	p_0 = \lim_{u \to 0} \frac{uf'(u)}{f(u)} = 
	\lim_{u \to 0}\left(p_1 - p_2\frac{u}{1+u}\right) = 
	p_1 > p_{JL} 
$$
and
$$	
	1 < p_{\infty} = \lim_{u \to \infty} \frac{uf'(u)}{f(u)} = p_1-p_2 < p_{JL}.
$$
By the relations $1/p_0 + 1/q_0 = 1$ and $1/p_{\infty} + 1/q_{\infty} = 1$, 
we obtain $q_0 < q_{JL}$ and $q_{\infty} > q_{JL}$. 
Thus (III) holds by Corollary \ref{cor1.7}.
\end{example}

%%%%%%%%%%%%%%%%%%%%%%%%%%%%%%
\bigskip

\noindent
{\bf Acknowledgements} 
The first author was supported by JSPS KAKENHI Grant Number 24K00530 
and the second author was supported by 
JSPS KAKENHI Grant Number 23K03167. 
This work was also supported by Research Institute for Mathematical Sciences, a Joint
Usage/Research Center located in Kyoto University.
\bigskip

%%%%%%%%%%%%%%%%%%%%%%%%%%%%%%%%%%%%%%%%%%%%%

%%%%%%%%%%%%%%%%%%%%%%%%%%%%%%%%%%%%%%%%%%%

\end{document}